%% file: CSL-arboretum-main.tex
\documentclass[a4paper,UKenglish,cleveref, autoref, thm-restate]{lipics-v2021}
\pdfoutput=1

\usepackage[utf8]{inputenc}
\usepackage[T1]{fontenc}
\usepackage{fontaxes}
\usepackage{microtype}
\usepackage{mathrsfs}
\usepackage{mathtools}
\usepackage[all]{xy}
\usepackage{tikz-cd}
\usepackage{relsize}
\usepackage{complexity}
\usepackage{xpatch}
\usepackage{graphicx}

\newcommand{\Cc}{\mathscr{C}}
\newcommand{\Dd}{\mathscr{D}}
\newcommand{\Ee}{\mathcal{E}}
\newcommand{\Gg}{\mathcal{G}}
\newcommand{\Hh}{\mathcal{H}}
\newcommand{\TP}{\mathcal{TP}}
\newcommand{\PW}{\mathcal{PW}}
\newcommand{\TW}{\mathcal{TW}}

\newcommand{\Ff}{\mathscr{F}}
\newcommand{\Ll}{\mathcal{L}}

\newcommand{\Tt}{\mathcal{T}}
\newcommand{\Pp}{\mathcal{P}}

\newcommand{\dist}{\mathrm{dist}}

\renewcommand{\phi}{\varphi}

\renewcommand{\FO}{{\rm FO}}
\newcommand{\MSO}{{\rm MSO}}

\newcommand{\cover}{\mathrel{\mathlarger{\vartriangleleft}}}
\newcommand{\rloc}[2]{\ensuremath{\mathcal{B}_{#1}^{#2}}}

\newtheorem{fact}{Observation}
\DeclareMathOperator{\join}{+}
\DeclareMathOperator{\comp}{\oplus}
\DeclareMathOperator{\union}{\cup}
\makeatletter
\xpatchcmd{\thmt@restatable}
{\csname #2\@xa\endcsname\ifx\@nx#1\@nx\else[{#1}]\fi}
{\ifthmt@thisistheone\csname #2\@xa\endcsname\ifx\@nx#1\@nx\else[{#1}]\fi\else\csname #2\@xa\endcsname\fi}
{}{} 
\makeatother
\newcommand{\blah}{proved in the appendix}

\title{Structural properties of the first-order transduction quasiorder} 

\titlerunning{Structural properties of the first-order transduction quasiorder} 

\author{Jaroslav Ne\v set\v ril}{Computer Science Institute of Charles University (IUUK), Praha, Czech Republic}{nesetril@iuuk.mff.cuni.cz}{https://orcid.org/0000-0002-5133-5586}{}

\author{Patrice Ossona de Mendez}{Centre d'Analyse et de Math\'ematiques Sociales (CNRS, UMR 8557),
Paris, France and\\
Computer Science Institute of Charles University,
  Praha, Czech Republic}{pom@ehess.fr}{https://orcid.org/0000-0003-0724-3729
}{}

\author{Sebastian Siebertz}{University of Bremen, Bremen, Germany}{siebertz@uni-bremen.de}{https://orcid.org/0000-0002-6347-1198}{}

\authorrunning{J.\ Ne\v{s}e\v{r}il, P.\ Ossona de Mendez and S.\ Siebertz} 

\Copyright{Jaroslav Ne\v{s}e\v{r}il, Patrice Ossona de Mendez and Sebastian Siebertz} 



\category{} 

\relatedversion{} 

\supplement{}

\funding{This paper is part of a project that has received funding from the European Research Council (ERC) under the European Union's Horizon 2020 research and innovation programme (grant agreement No 810115 -- {\sc Dynasnet}) and from the German
Research Foundation (DFG) with grant agreement
No 444419611. \\\includegraphics[width=\textwidth/4]{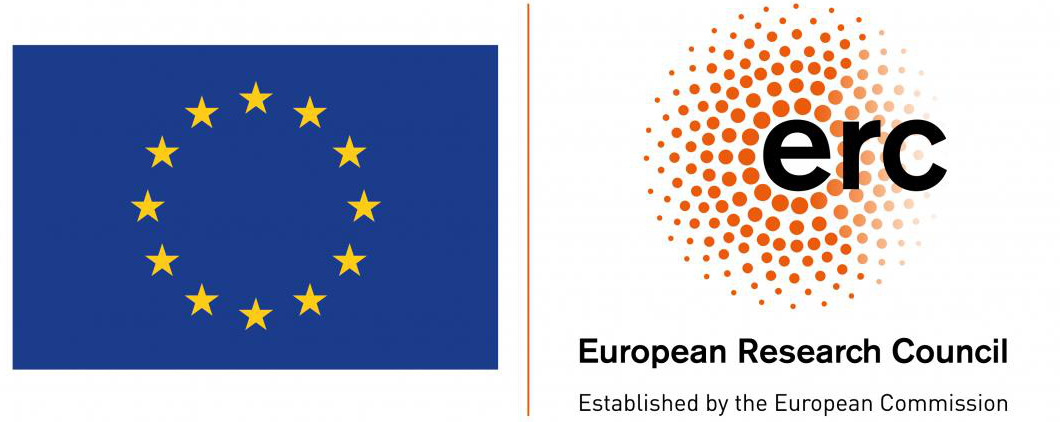}\hspace{2pt}\includegraphics[width=\textwidth/5]{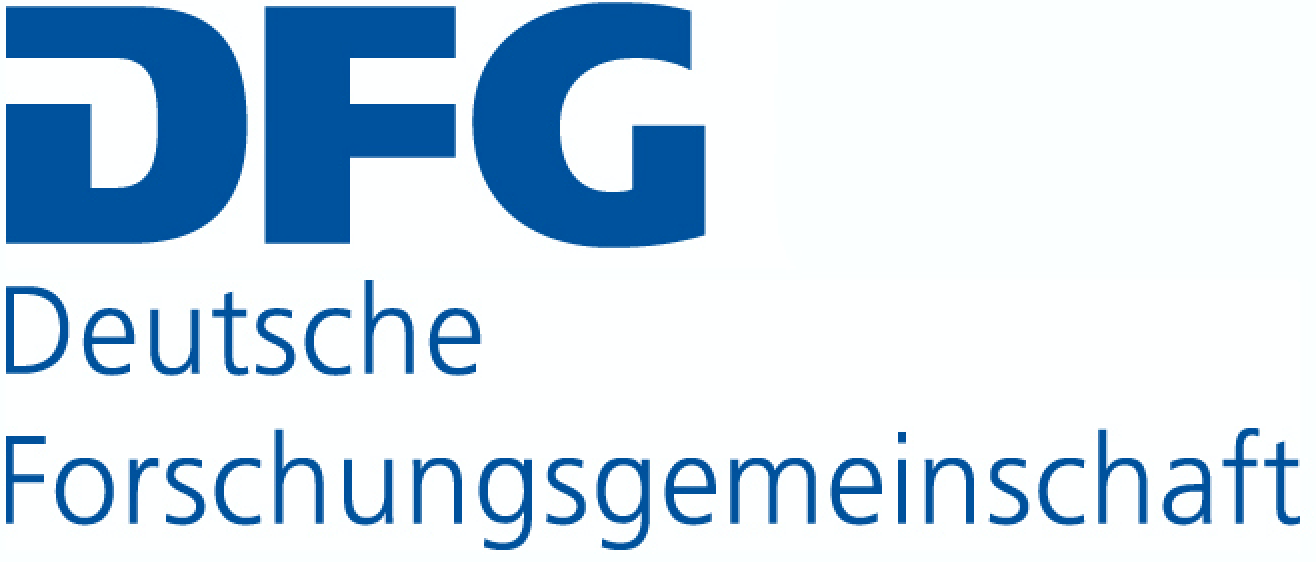}}

\acknowledgements{}

\nolinenumbers 

\hideLIPIcs  

\EventEditors{John Q. Open and Joan R. Access}
\EventNoEds{2}
\EventLongTitle{42nd Conference on Very Important Topics (CVIT 2016)}
\EventShortTitle{CVIT 2016}
\EventAcronym{CVIT}
\EventYear{2016}
\EventDate{December 24--27, 2016}
\EventLocation{Little Whinging, United Kingdom}
\EventLogo{}
\SeriesVolume{42}
\ArticleNo{23}

\begin{document}

\maketitle
\input{abstract-new}

\input{intro-seb}

\input{preliminaries}

\input{trans}
\input{path}

%
\input{local}

\input{order}

\clearpage
\bibliography{ref}
\clearpage

\pagebreak
\appendix
\input{monotone}

\end{document}

%% file: abstract-new.tex
\begin{abstract}
  Logical transductions provide a very useful tool to encode
  classes of structures inside other classes of structures.  In this
  paper we study first-order (FO) transductions and the quasiorder they 
  induce on infinite classes of finite graphs. 
  Surprisingly, this quasiorder is very complex, though shaped by 
  the locality properties of first-order logic. This contrasts with 
  the conjectured simplicity of the monadic second order (MSO) transduction
  quasiorder. 
  We first establish a local normal form for FO transductions, which is of independent interest.
  Then we prove that the quotient partial order is a bounded distributive join-semilattice, and that the subposet of \emph{additive} classes is also a bounded distributive join-semilattice.
 The FO transduction quasiorder has a great expressive power, and many well studied class properties can be defined using it.
  We apply these structural properties to prove, among other results, 
  that FO transductions of the class of paths are exactly perturbations of classes with bounded bandwidth, that the local variants of monadic stability and monadic dependence are equivalent to their (standard) non-local versions, and that the classes with pathwidth at most $k$, for $k\geq 1$ form a strict hierarchy in the FO transduction quasiorder.
 	\end{abstract}

%% file: intro-seb.tex
\vspace{-2mm}
\section{Introduction and statement of results}

Transductions provide a model theoretical tool to encode relational
structures (or classes of relational structures) inside other 
(classes of) relational structures. Transductions 
naturally induce a quasiorder, that is, a reflexive and
transitive binary relation, on classes of relational structures. 
We study here the first-order~(\FO) and
monadic second-order~(\MSO) transduction quasiorders $\sqsubseteq_\FO$
and~$\sqsubseteq_\MSO$ on infinite classes of finite graphs. 
These quasiorders are very different and both
have a sound combinatorial and model theoretic relevance, as 
we will outline below. To foster the further discussion, let us
(slightly informally) introduce the concept of transductions. 
Formal definitions will be given in \cref{sec:prelims}. 

A \emph{transduction} $\mathsf T$ (on graphs) is the composition of a copying operation, a coloring operation, and a simple interpretation. The \emph{copying operation} $\mathsf C_k$ maps a graph $G$ to the graph~$\mathsf C_k(G)$ obtained 
	by taking $k$ disjoint copies of $G$ and making all the copies of a single vertex adjacent; the \emph{coloring operation} maps a graph $G$ to the set $\Gamma(G)$ of all possible colorings of~$G$; a \emph{simple interpretation} $\mathsf I$ maps a colored graph $G^+$ to a graph $H$, whose vertex set (resp.\ edge set) is a  definable subset of $V(G^+)$ (resp.\ of $V(G^+)\times V(G^+)$).
In this way, the transduction $\mathsf T$ maps a graph $G$ to a set 
$\mathsf T(G)$ of graphs defined as $\mathsf T(G)\coloneqq \mathsf I\circ\Gamma\circ\mathsf C_k(G)=\{\mathsf I(H^+): H^+\in\Gamma(\mathsf C_k(G))\}$.
This naturally extends to a graph class $\Cc$ by $\mathsf T(\Cc)\coloneqq \bigcup_{G\in\Cc}\mathsf T(G)$.

We say that a class $\Cc$ \emph{is a transduction of} a class $\Dd$ if there exists a transduction $\mathsf T$ with $\Cc\subseteq\mathsf T(\Dd)$, and we denote this by $\Cc\sqsubseteq\Dd$. We write  $\Cc\equiv\Dd$
for~$\Cc\sqsubseteq\Dd$ and~$\Dd\sqsubseteq\Cc$, $\Cc\sqsubset\Dd$ for
$\Cc\sqsubseteq\Dd$ and $\Cc\not\equiv\Dd$, and $\Cc\cover\Dd$
for the property that $(\Cc,\Dd)$ is a \emph{cover}, that is,
that~$\Cc\sqsubset\Dd$ and there is no class $\Ff$ with
$\Cc\sqsubset\Ff\sqsubset\Dd$. For a logic $\Ll$ we write 
$\Cc\sqsubseteq_\Ll\Dd$ to stress that the simple interpretation of 
the transduction uses $\Ll$-formulas. 
 
For most commonly studied logics $\Ll$ transductions compose and in this
case $\sqsubseteq_\Ll$ is a quasiorder. We study here mainly the first-order~(\FO) and
monadic second-order~(\MSO) transduction quasiorders $\sqsubseteq_\FO$
and~$\sqsubseteq_\MSO$.  
As with the colorings all vertex subsets become definable, it 
follows that we can restrict our attention to
infinite \emph{hereditary} classes, that is, infinite classes that are
closed under taking induced subgraphs.

{\MSO} transductions are basically understood.  Let us write $\Ee$ for
the class of edgeless graphs, $\Tt_n$~for the class of forests of
depth $n$ (where the depth of a (rooted) tree is the maximum number of 
vertices on a root-leaf path, hence $\Tt_1=\Ee$), 
$\Pp$~for the class of all paths, $\Tt$ for the class of
all trees and $\Gg$ for the class of all graphs. The {\MSO} transduction
quasiorder is conjectured to be simply the chain
$\Ee\cover_\MSO
\Tt_2\cover_\MSO\ldots\cover_\MSO\Tt_n\cover_\MSO\ldots\sqsubseteq_\MSO\Pp\cover_\MSO\Tt\cover_\MSO~\Gg$~\cite{blumensath10}.

In a combinatorial setting this hierarchy has a very concrete meaning
and it was investigated using the following notions: a class $\Cc$ has
\emph{bounded shrubdepth} if $\Cc\sqsubseteq_\MSO \Tt_n$ for some $n$;
$\Cc$ has \emph{bounded linear cliquewidth} if
$\Cc\sqsubseteq_\MSO \Pp$; $\Cc$ has \emph{bounded cliquewidth}
if~$\Cc\sqsubseteq_\MSO \Tt$. These definitions very nicely illustrate
the treelike structure of graphs from the above mentioned classes from
a logical point of view, which is combinatorially captured by the
existence of treelike decompositions with certain properties. It is
still open whether the {\MSO} transduction quasiorder is as shown above
\cite[Open Problem 9.3]{blumensath10}, though the initial fragment
$\Ee\cover_\MSO
\Tt_2\cover_\MSO\Tt_3\cover_\MSO\ldots\cover_\MSO\Tt_n$ has been
proved to be as stated in~\cite{Ganian2017}. Thus we are essentially left
with the following three questions: Does $\Cc\sqsubset_\MSO\Pp$ imply
$(\exists n)\ \Cc\sqsubseteq_\MSO\Tt_n$? This is equivalent to the
question whether one can transduce with {\MSO} arbitrary long paths from
any class of unbounded shrubdepth (see \cite{kwon19} for a proof of
the CMSO version). Is the pair $(\Pp,\Tt)$ a cover? Is the pair
$(\Tt,\Gg)$ a cover? This last question is related to a famous
conjecture of Seese~\cite{seese1991structure} and the CMSO version has
been proved in~\cite{courcelle2007vertex}.


\smallskip
As in the {\MSO} case, the {\FO} transduction quasiorder
allows to draw important algorithmic and structural dividing
lines. For instance {\MSO} collapses to {\FO} on classes of bounded
shrubdepth~\cite{Ganian2017}. Classes of bounded shrubdepth are
also characterized has being {\FO} transductions of classes
of trees of bounded depth~\cite{Ganian2012}.
{\FO} transductions give alternative
characterizations of other graph class properties mentioned above:
a class $\Cc$ has bounded linear cliquewidth if and only if
  $\Cc\sqsubseteq_\FO\Hh$, where $\Hh$ denotes the class of
  \emph{half-graphs} (bipartite graphs with vertex set
  $\{a_1,\ldots, a_n\}\, \cup\, \{b_1,\ldots, b_n\}$ and edge set
  $\{a_ib_j~:~1\leq i\leq j\leq n\}$ for some~$n$)~\cite{colcombet2007combinatorial}, and
bounded cliquewidth if and only if $\Cc\sqsubseteq_\FO\TP$,
  where $\TP$ denotes the class of \emph{trivially perfect graphs}
  (comparability graphs of rooted trees)
  \cite{colcombet2007combinatorial}.
%
Also, it follows from ~\cite{baldwin1985second} that {\FO} transductions allow to give an alternative characterizations
of classical model theoretical properties:
A class $\Cc$ is \emph{monadically
  stable} if $\Cc\not\sqsupseteq_\FO \Hh$ 
  and \emph{monadically
  dependent} if~$\Cc\not\sqsupseteq_\FO \Gg$. 
%
We further call a class $\Cc$ \emph{monadically straight} if
$\Cc\not\sqsupseteq_\FO \TP$. To the best of our knowledge this
property has not been studied in the literature but seems to play a
key role in the study of {\FO} transductions.

The {\FO} transduction quasiorder has not been studied in detail previously 
and it turns out that it is much more complicated than the {\MSO}
transduction quasiorder. This is outlined in \cref{fig:hierarchy},
and it is the goal of this paper to explore this quasiorder.

\begin{figure}[t]
\hfill
\includegraphics[width=.8\textwidth]{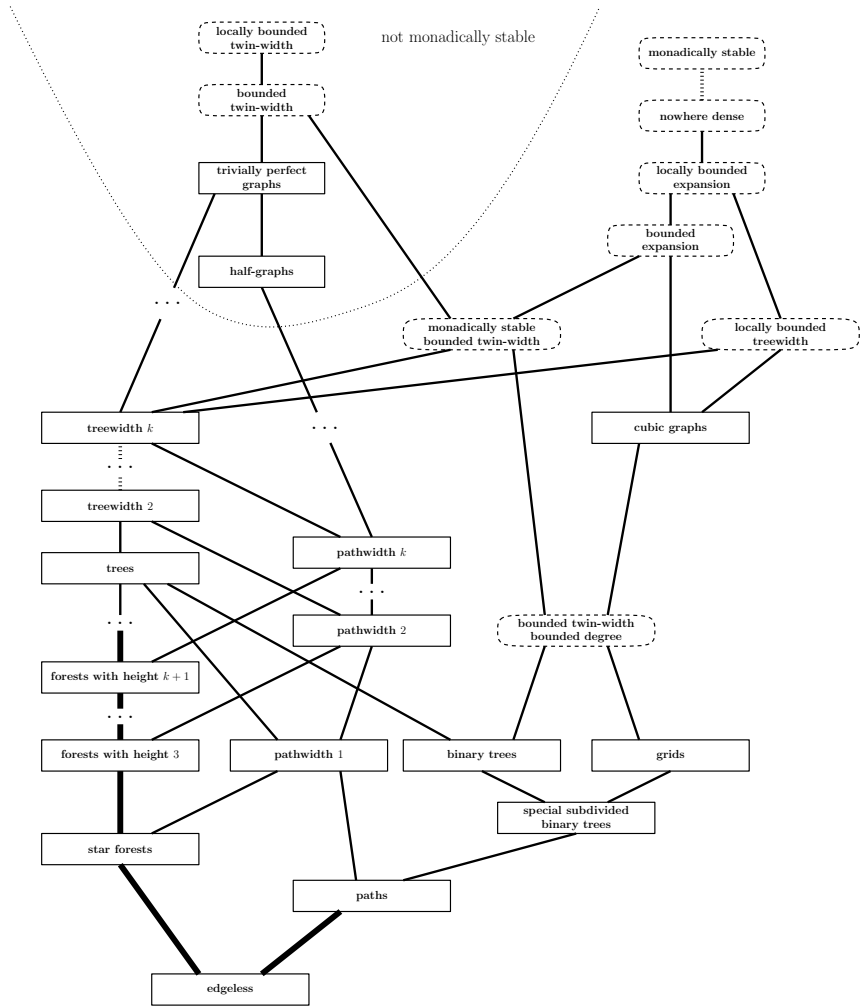}
\hfill\hfill
\caption{Partial outline of the {\FO} transduction quasiorder. The
  special subdivided binary trees are those subdivisions of binary
  trees that are subgraphs of the grid.  Dashed boxes correspond to
  families of not necessarily transduction equivalent graph classes
  sharing a common property. Fat lines correspond to covers, normal
  lines correspond to strict containment $\sqsubset$, dotted lines correspond to
  containment (with a possible collapse).}
\label{fig:hierarchy}
\end{figure}
We are motivated by three aspects of the $\sqsubseteq_\FO$
quasiorder that have been specifically considered in the past and
appeared to be highly non-trivial.  
The first aspect is the conjectured property that every class that cannot {\FO} transduce paths has bounded shrubdepth (hence is an {\FO] transduction of a class of bounded height trees).
%
%
%
The second aspect that was studied in detail concerns the chain formed
by classes with bounded pathwidth, which is eventually covered by the
class of half-graphs. This is related to the fact that 
%
  in the {\FO} transduction quasiorder there is no class between
  the classes with bounded pathwidth and the class $\Hh$ of half-graphs~\cite{SODA_msrw, msrw}.
The third aspect concerns the chain of classes with bounded treewidth,
which is eventually covered by the class of trivially perfect
graphs. This is related to the fact that 
%
if $\Hh\not\sqsubseteq_\FO\Cc\sqsubseteq_\FO\TP$ (that is, $\Cc$ is
a monadically stable class with bounded cliquewidth), then
$\Cc\sqsubseteq_\FO\TW_n$ for some $n$, where $\TW_n$ denotes the
class of graphs with treewidth at most~$n$~\cite{nevsetril2020rankwidth}.

\smallskip

In this paper, we establish the three kinds of results and show that despite its complexity the FO transduction
 quasiorder is strongly structured.

\paragraph*{A local normal form for transductions}
In \cref{sec:norm} we introduce a normal form for {\FO} transductions
that captures the local character of first-order logic, by proving
that every {\FO} transduction can be written as the composition
of a copying operation, a transduction that connects only vertices
at a bounded distance, and a perturbation, which is a sequence of
subset complementations (\cref{thm:normal}). 
In \cref{sec:applications-nf} we give two applications of this 
normal form. We first 
characterize the equivalence class of the class of paths in the 
{\FO} transduction quasiorder (\cref{thm:transpath}). Then, 
we prove that the local versions of monadic stability, 
monadic straightness, and monadic dependence are 
equivalent to the non-local versions (\cref{thm:local}). 
This result is of independent interest and may be relevant
e.g.\ for locality based FO model-checking on these
classes. 

\paragraph*{Structural properties of the transduction quasiorder}
In \cref{sec:structural-prop} we prove that the partial orders
obtained  as the quotient of the transduction quasiorder and the 
non-copying transduction quasiorder are bounded distributive 
join-semilattices (\cref{thm:sl}) and discuss some of their properties. 
In particular we prove that every class
closed under disjoint union is join-irreducible. 
Recall that a partial order $(X,\leq)$
is a \emph{join semi-lattice} if for all $x,y\in X$ there exists
a least upper bound $x\vee y$ of $\{x,y\}$, called the join of $x$ and $y$. 
It is \emph{distributive} if, for all $a,b,x\in X$ with $x\leq a\vee b$ there exist $x_1\leq a$, $x_2\leq b$, with $x=x_1\vee x_2$.
An element $x\in X$  is \emph{join-irreducible} if $x$ is not the join of two incomparable elements. 
Then we consider the subposets induced by \emph{additive} classes, which are the classes equivalent to the class of disjoint unions of pairs of graphs in the class. 
We prove that these subposets are also 
bounded distributive join-semilattices (\cref{thm:slA}), but with a 
different join.  We discuss
some properties of these subposets and in particular prove that 
every class closed under disjoint union and 
equivalent to its subclass of connected graphs is join-irreducible. 

\paragraph*{The transduction quasiorder on some classes}
In \cref{sec:special} we focus on the transduction quasiorders
on the class of paths, the class of trees, 
classes of bounded height trees, classes with bounded
pathwidth, classes with bounded treewidth, and derivatives. 
In particular we prove that classes with bounded pathwidth 
form a strict hierarchy (\cref{thm:PW}). 
This
result  was the main motivation for this study, and we conjecture
that a similar statement holds with treewidth. This would be a
consequence of the conjecture that the class of all graphs with treewidth at most $n$ is
  incomparable with the class of all graphs with pathwidth at most $n+1$, for every positive integer $n$.

%% file: preliminaries.tex
\section{Preliminaries and basic properties of transductions}\label{sec:prelims}
We assume familiarity with first-order logic and graph theory and
refer e.g.\ to \cite{diestel2012graph,hodges1993model} for background
and for all undefined notation.  The vertex set of a graph $G$ is
denoted as~$V(G)$ and its edge set $E(G)$. The {\em complement} of a
graph $G$ is the graph $\overline{G}$ with the same vertex set, in
which two vertices are adjacent if they are not adjacent in $G$. The
disjoint union of two graphs $G$ and $H$ is denoted as $G\union H$, and
their {\em complete join} $\overline{\overline{G}\union \overline{H}}$ as
$G\join H$. We denote by $K_t$ the complete graph on $t$
vertices. Hence, $G\join K_1$ is obtained from $G$ by adding a 
new vertex, called an \emph{apex}, that is connected to all vertices of 
$G$. For a class $\Cc$ of graphs we denote by $\Cc\join K_1$ the 
class obtained from $\Cc$ by adding an apex to each graph of $\Cc$.
The {\em lexicographic product} $G\bullet H$ of two
graphs $G$ and $H$ is the graph with vertex set $V(G)\times V(H)$, in
which~$(u,v)$ is adjacent to $(u',v')$ if either $u$ is adjacent 
to~$u'$ in~$G$ or~$u=u'$ and~$v$ is adjacent to~$v'$ in $H$.  The {\em
  pathwidth} ${\rm pw}(G)$ of a graph $G$ is equal to one less than
the smallest clique number of an interval graph that contains $G$ as a
subgraph, that is,
${\rm pw}(G)=\min\{\omega(H)-1: \text{for an interval graph }H\text{
  with }H\supseteq G\}$. The {\em treewidth} ${\rm tw}(G)$ of a
graph~$G$ is equal to one less than the smallest clique number of a
chordal graph that contains~$G$ as a subgraph, that is,
\mbox{${\rm tw}(G)=\min\{\omega(H)-1: \text{for a chordal graph
  }H\text{ with }H\supseteq G\}$}.  We write $G^k$ for the {\em $k$-th
  power} of $G$ (which has the same vertex set as $G$ and two vertices
are connected if their distance is at most $k$ in $G$).  The {\em
  bandwidth} of a graph $G$ is
${\rm bw}(G)=\min\{\ell:\text{for }P\in\Pp\text{ with }P^\ell\supseteq
G, \}$.

In this paper we consider either graphs or {\em $\Sigma$-expanded
  graphs}, that is, graphs with additional unary relations in $\Sigma$
(for a set $\Sigma$ of unary relation symbols). We usually denote
graphs by $G,H,\ldots$ and $\Sigma$-expanded graphs by
$G^+,H^+,G^*,H^*,\ldots$, but sometimes we will use $G,H,\ldots$ for
$\Sigma$-expanded graphs as well. We shall often use the term ``colored graph'' instead of $\Sigma$-expanded graph.
 In formulas, the adjacency relation
will be denoted as $E(x,y)$. For each non-negative integer $r$ we can
write a formula $\delta_{\leq r}(x,y)$ such that for every graph~$G$
and all $u,v\in V(G)$ we have $G\models\delta_{\leq r}(u,v)$ if and
only if the distance between~$u$ and~$v$ in $G$ is at most $r$. For
improved readability we write $\dist(x,y)\leq r$
for~$\delta_{\leq r}(x,y)$.  For $U\subseteq V(G)$ we write $B_r^G(U)$
for the subgraph of $G$ induced by the union of the closed
$r$-neighborhoods of the vertices in $U$. We write~$N^G(v)$ for the
open neighborhood of $v$ (as a set of vertices).  For the sake of
simplicity we use for balls of radius $r$ the notation $B_r^G(v)$ instead of $B_r^G(\{v\})$
and, if~$G$ is clear from the context, we drop the superscript
$G$. For a class~$\Cc$ and an integer $r$, we denote by
$\rloc{r}{\Cc}$ the class of all the balls of radius $r$ of graphs in~$\Cc$: $ \rloc{r}{\Cc}=\{B_r^G(v)\mid G\in\Cc\text{ and }v\in V(G)\}$.
%
For a formula $\phi(x_1,\ldots, x_k)$ and a graph (or a
$\Sigma$-expanded graph) $G$ we define
\[\phi(G)\coloneqq \{(v_1,\ldots, v_k)\in
V(G)^k~:~G\models\phi(v_1,\ldots, v_k)\}.\]

%

%
%

For a positive integer $k$, the \emph{$k$-copy operation}
$\mathsf{C}_k$ maps a graph $G$ to the graph~$\mathsf{C}_k(G)$
consisting of $k$ copies of $G$ where the copies of each vertex are
made adjacent (that is, the copies of each vertex induce a clique and
there are no other edges between the copies of $G$). Note that for $k=1$, $\mathsf C_k$ maps each graph $G$ to itself. (Thus $\mathsf C_1$ is the identity mapping.)

For a set $\Sigma$ of unary relations, the \emph{coloring operation} $\Gamma_\Sigma$ maps a graph $G$ to the set $\Gamma_\Sigma(G)$ of all its $\Sigma$-expansions.

A \emph{simple  interpretation} $\mathsf I$ of graphs in \mbox{$\Sigma$-expanded}
graphs is a pair $(\nu(x), \eta(x,y))$ consisting of two formulas (in
the language of $\Sigma$-expanded graphs), where $\eta$ is symmetric
and anti-reflexive (i.e.
\mbox{$\models \eta(x,y)\leftrightarrow\eta(y,x)$} and
\mbox{$\models \eta(x,y)\rightarrow\neg(x=y)$}). If $G^+$ is a
$\Sigma$-expanded graph, then $H=\mathsf I(G^+)$ is the graph with
vertex set $V(H)=\nu(G^+)$ and edge set $E(H)=\eta(G)\cap \nu(G)^2$.

A \emph{transduction} $\mathsf T$ is the composition
$\mathsf I\circ\Gamma_\Sigma\circ \mathsf C_k$ of a copy operation $\mathsf C_k$, a coloring operation~$\Gamma_\Sigma$, and a simple interpretation $\mathsf I$ of graphs in $\Sigma$-expanded graphs.
In other words, for every graph $G$ we have
$\mathsf T(G)=\{\mathsf I(H^+): H\in\Gamma_\Sigma(\mathsf C_k(G))\}$.
A transduction $\mathsf T$ is \emph{non-copying} if it is the composition of a coloring operation and a simple interpretation, that is if it can written as 
$\mathsf I\circ\Gamma_\Sigma\circ \mathsf C_1$ ($=\mathsf I\circ\Gamma_\Sigma$).
We say that a
transduction $\mathsf T'$ \emph{subsumes} a transduction $\mathsf T$
if for every graph $G$ we have
$\mathsf T'(G)\supseteq\mathsf T(G)$. We denote by
$\mathsf T'\geq\mathsf T$ the property that $\mathsf T'$ subsumes
$\mathsf T$.

For a class $\Dd$ and a transduction $\mathsf T$ we define $\mathsf T(\Dd)=\bigcup_{G\in\Dd}\mathsf T(G)$ and 
we say that a class~$\Cc$ is a  {\em $\mathsf T$-transduction} of $\Dd$ if $\Cc\subseteq\mathsf T(\Dd)$. We also say that $\mathsf{T}$ \emph{encodes} $\Cc$ in $\Dd$. 
A class~$\Cc$ of graphs is a \emph{(non-copying) transduction} of a class $\Dd$ of
graphs if it is a $\mathsf T$-transduction of~$\Dd$ for some (non-copying) transduction $\mathsf T$.
We denote by $\Cc\sqsubseteq_\FO\Dd$ (resp. $\Cc\sqsubseteq^{\circ}_\FO\Dd$) the property that the class $\Cc$ is an FO transduction (resp. a non-copying FO transduction) of the class $\Dd$.
It is easily checked that the composition of two (non-copying) transductions is a (non-copying) transduction (see, for instance \cite{SBE_drops}). Thus the relations
$\Cc\sqsubseteq_\FO\Dd$ and 
$\Cc\sqsubseteq^{\circ}_\FO\Dd$ are quasiorders
on classes of graphs
Intuitively, if $\Cc\sqsubseteq_\FO \Dd$, then $\Cc$ is at most as complex as $\Dd$.
Equivalences for $\sqsubseteq_\FO$ and $\sqsubseteq_\FO^0$ are defined naturally.

We say that a class $\Cc$ \emph{does not need copying} if for every
integer $k$ the class $\mathsf{C}_k(\Cc)$ is a non-copying
transduction of $\Cc$.   

We take time for some observations.

\begin{fact}
	If $\Cc$ does not
need copying and $\Cc\equiv_\FO\Dd$, then $\Dd$ does not need
copying.
\end{fact}
\begin{proof}
	This follows from the fact that every class $\Cc$ is  a non-copying
transduction of~$\mathsf C_k(\Cc)$. 
\end{proof}

\begin{fact}
	A class $\Cc$ does not need copying if and only if $\mathsf C_2(\Cc)\equiv_\FO^\circ \Cc$.
\end{fact}
\begin{proof}
It is easily checked that for every positive integer $k$ there is a non-copying transduction~$\mathsf T_k$ such that $\mathsf T_k\circ\mathsf C_k\circ\mathsf C_2$ subsumes $\mathsf C_{2k}$. 
Assume $\mathsf C_2(\Cc)\equiv_\FO^\circ \Cc$.
Then if $\mathsf C_k(\Cc)\sqsubseteq_\FO^\circ \Cc$ we deduce from $\mathsf C_2(\Cc)\sqsubseteq_\FO^\circ \Cc$ that $\mathsf C_{2k}(\Cc)\sqsubseteq_\FO^\circ \Cc$. By induction we get $\mathsf C_k(\Cc)\equiv_\FO^\circ \Cc$ for every positive integer $k$.
\end{proof}

%
\begin{fact}\label{fact:non-copy}
  A class $\Dd$ does not need copying if and only if for every class $\Cc$ we have
  $\Cc\sqsubseteq_\FO\Dd$ if and only if
  $\Cc\sqsubseteq_\FO^\circ\Dd$.
\end{fact}

\begin{fact}\label{fact:non-copy-classes}
 If a class $\Cc$ is closed under adding pendant vertices (that is, if $G\in \Cc$
and $v\in V(G)$, then $G'$, which is obtained from~$G$ by adding a new
vertex adjacent only to $v$, is also in $\Cc$) then $\Cc$ does not need copying.
\end{fact}

%
A {\em subset complementation} transduction is defined by the
quantifier-free interpretation on a $\Sigma$-expansion (with
$\Sigma=\{M\}$) by
$\eta(x,y):=\neg \bigl(E(x,y)\leftrightarrow(M(x)\wedge M(y)\bigr)$.
In other words, the subset complementation transduction complements
the adjacency inside the subset of the vertex set defined by $M$. We
denote by $\comp M$ the subset complementation defined by the unary
relation $M$. A {\em perturbation} is a composition of (a bounded
number of) subset complementations.
Let $r$ be a non-negative integer.  A formula $\phi(x_1,\ldots, x_k)$
is \emph{$r$-local} if for every ($\Sigma$-expanded) graph $G$ and all
$v_1,\ldots, v_k\in V(G)$ we have
$G\models\phi(v_1,\ldots, v_k) \iff 
  B_r^G(\{v_1,\ldots, v_k\})\models\phi(v_1,\ldots, v_k)$.
An $r$-local formula $\phi(x_1,\ldots, x_k)$ is \emph{strongly
  $r$-local} if
$\models\phi(x_1,\ldots,x_k)\rightarrow \dist(x_i,x_j)\leq r$ for all
$1\leq i<j\leq k$ (see \cite{Loclim}).

\begin{lemma}[Gaifman's Locality Theorem~\cite{gaifman1982local}]\label{lem:gaifman}
  Every formula $\phi(x_1,\ldots, x_m)$ is equivalent to a Boolean
  combination of $t$-local formulas 
  and so-called \emph{basic local sentences} of the form
\[\exists x_1\ldots\exists x_k \big(\bigwedge_{1\leq i\leq k}\chi(x_i)
\wedge\bigwedge_{1\leq i<j\leq k}\dist(x_i,x_j)>2r\quad 
\text{(where $\chi$ is $r$-local).}\] 
Furthermore, if the quantifier-rank of $\phi$ is $q$, then 
$r\leq 7^{q-1}$, $t\leq 7^{q-1}/2$, and $k\leq q+m$. 
\end{lemma}

We call a transduction $\mathsf T$ {\em immersive} if it is
non-copying and the formulas in the interpretation associated to
$\mathsf T$ are strongly local.

%% file: trans.tex
\section{Local properties of FO transductions}
\subsection{A local normal form}
\label{sec:norm}
We now establish a normal form for first-order transductions that
captures the local character of first-order logic and further
study the properties of immersive transductions. The normal form is
based on Gaifman's Locality Theorem and uses only \emph{strongly local
  formulas}, while the basic-local sentences are handled by subset
complementations. This normal form will be one of the main tools to 
establish results in the paper.  

\begin{theorem}
\label{thm:normal}
Every non-copying transduction $\mathsf T$ is subsumed by the composition of an immersive transduction $\mathsf T_{\rm imm}$ and a perturbation~$\mathsf P$, that is $\mathsf T\leq \mathsf P\circ\mathsf T_{\rm imm}$.

Consequently, every transduction $\mathsf T$ is subsumed by the composition of a
copying operation~$\mathsf C$, an immersive transduction $\mathsf T_{\rm imm}$ and a perturbation~$\mathsf P$, that is $\mathsf T\leq \mathsf P\circ\mathsf T_{\rm imm}\circ\mathsf C$.
\end{theorem}



\begin{proof}
	Let $\mathsf T=\mathsf I_{\mathsf T}\circ\Gamma_{\Sigma_{\mathsf T}}$ be a non-copying transduction.
	Without loss of generality, we may assume that the interpretation $\mathsf I_{\mathsf T}$ defines the  domain directly from the $\Sigma_{\mathsf T}$-expansion. Then the only non-trivial part of the interpretation is the adjacency relation, which is defined by a symmetric and anti-reflexive formula $\eta(x,y)$.	
	We shall prove that the transduction $\mathsf T$ is subsumed by the composition of an immersive transduction $\mathsf T_\psi$ and a perturbation $\mathsf P$.
	
	We define $\Sigma_{\mathsf T_\psi}$ as the disjoint union of $\Sigma_{\mathsf T}$ and a set $\Sigma_\psi=\{T_i\mid 1\leq i\leq n_1\}$ for some integer $n_1$ we shall specify later and 
	let $\Sigma_{\mathsf P}=\{Z_j\mid 1\leq j\leq n_2\}$ for some integer $n_2$ we shall also specify later.
  Let $q$ be the quantifier rank of $\eta(x,y)$.  According to
  \cref{lem:gaifman}, $\eta$ is logically equivalent to a formula in 
  Gaifman normal form, that is, to a Boolean
  combination of $t$-local formulas and
  basic-local sentences~$\theta_1,\dots,\theta_{n_1}$.  To each
  $\theta_i$ we associate a unary predicate $T_i\in \Sigma_{\psi}$. 
  We consider the
  formula $\widetilde{\eta}(x,y)$ obtained from the Gaifman normal form of $\eta(x,y)$ 
  by replacing the sentence $\theta_i$ by the atomic formula $T_i(x)$. Note that $\widetilde{\eta}$ is $t$-local. 
  
  Under the assumption that ${\rm dist}(x,y)>2t$
  every $t$-local formula $\chi(x,y)$ is equivalent to $\chi_1(x)\wedge
  \chi_2(y)$ for $t$-local formulas $\chi_1(x)$ and $\chi_2(y)$. Furthermore, 
  $t$-local formulas are closed under boolean combinations. 
  By bringing $\widetilde{\eta}$ into
  disjunctive normal form and grouping conjuncts appropriately, 
  it follows that under the assumption ${\rm dist}(x,y)>2t$ the
  formula $\widetilde{\eta}$ is equivalent to a formula $\widetilde{\phi}(x,y)$ of the form
  $\bigvee_{(i,j)\in\mathcal F}\zeta_i(x)\wedge\zeta_j(y)$, where
  $\mathcal F\subseteq [n_2]\times[n_2]$ for some integer $n_2$ and the
  formulas $\zeta_i$ ($1\leq i\leq n_2$) are $t$-local. By considering appropriate 
  boolean combinations (or, for those familiar with model theory, by assuming that the $\zeta_i$ define local types)
  we may assume that $\models \forall x \bigwedge_{i\neq j} \neg(\zeta_i(x)\wedge
  \zeta_j(x))$, that is, every element of a graph satisfies at most one of the $\zeta_i$. 
  Note also that $\mathcal F$ is symmetric as $\eta$ (hence $\widetilde{\eta}$ and $\widetilde{\phi}$) are symmetric.
  
  We define 
  \mbox{$\psi(x,y):=\neg(\widetilde{\eta}(x,y)\leftrightarrow\widetilde{\phi}(x,y))\wedge ({\rm dist}(x,y)\leq 2t)$}, which is $2t$-strongly local, 
  and we define  $\mathsf I_{\mathsf T_\psi}$ as the interpretation of graphs in $\Sigma_{\mathsf T_\psi}$-structures by using the same definitions as in $\mathsf I_{\mathsf T}$ for the domain, then defining the adjacency relation by $\psi(x,y)$.  
  To each formula~$\zeta_i$ we associate a unary 
  predicate $Z_i\in \Sigma_{\mathsf P}$. We define the perturbation~$\mathsf{P}$
  as the sequence of subset complementations $\comp Z_i$ 
  (for~$(i,i)\in\mathcal F$) and of \mbox{$\comp Z_i\comp Z_j\comp (Z_i\cup Z_j)$} 
  (for~$(i,j)\in\mathcal F$ and~$i<j$). Denote by $\phi(x,y)$ the
  formula defining the edges in the interpretation~$I_\mathsf{P}$. Note that
  when the $Z_i$ are pairwise disjoint, then $\mathsf{P}$ complements
  exactly the edges of $Z_i$ or between~$Z_i$ and~$Z_j$, respectively. 
  The operation $\comp (Z_i\cup Z_j)$ complements all edges between
  $Z_i$ and $Z_j$, but also inside $Z_i$ and $Z_j$, which is undone by
  $\comp Z_j$ and~$\comp Z_i$.

%

	Now assume that a graph $H$ is a $\mathsf T$-transduction of a graph $G$, and let $G^+$ be a \mbox{$\Sigma_{\mathsf T}$-expansion} of $G$ such that $H=\mathsf I_{\mathsf T}(G^+)$.
	We define the $\Sigma_\psi$-expansion $G^*$ of $G^+$ (which is thus a $\Sigma_{\mathsf T_{\psi}}$-expansion of $G$)
	 by defining, for each $i\in[n_1]$, $T_i(G^*)=V(G)$ if
  $G^+\models \theta_i$ and $T_i(G^*)=\emptyset$ otherwise.
  	Let $K=\mathsf I_{\mathsf T_\psi}(G^*)$. We define the $\Sigma_{\mathsf P}$-expansion $K^+$ of $K$ by defining, for each $j\in [n_2]$, 
  $Z_j(K^+)=\zeta_j(G^+)$. By the assumption that $\models \forall x \bigwedge_{i\neq j} \neg(\zeta_i(x)\wedge
  \zeta_j(x))$ the~$Z_j$ are pairwise disjoint. 
  Now, when $\dist(x,y)>2t$ there is no edge between~$x$ and~$y$ in~$K$, 
  hence $\phi$ on $K^+$ is equivalent to~$\widetilde{\phi}$ on $G^*$, which 
  in turn in this case is equivalent to~$\widetilde{\eta}(x,y)$ on~$G^*$. 
  On the other hand, when $\dist(x,y)\leq 2t$, then the perturbation 
  is applied to the edges defined by $\neg(\widetilde{\eta}(x,y)\leftrightarrow\widetilde{\phi}(x,y)$, which yields exactly the edges defined by $\widetilde{\eta}$
  on~$G^*$. Thus we have $\eta(G^+)=\widetilde{\eta}(G^*)=\phi(K^+)$, 
  hence    $\mathsf I_{\mathsf P}(K^+)=H$.
  
	
	It follows that the transduction $\mathsf T$ is subsumed by the composition
  of the immersive transduction $\mathsf T_\psi$ and a sequence of 
  subset complementations, the perturbation $\mathsf P$.
\end{proof}

\begin{corollary}
\label{cor:norm}
For every immersive transduction $\mathsf T$ and every perturbation $\mathsf P$, there exist
immersive transduction $\mathsf T'$ and a perturbation $\mathsf P'$, such that 
$\mathsf P'\circ\mathsf T'$ subsumes $\mathsf T\circ\mathsf P$.
%
\end{corollary}

\subsection{Immersive transductions}
\label{sec:immersive}

Intuitively, copying operations and perturbations are simple operations. 
The main complexity of a transduction is captured by its immersive part. 
The strongly local character of immersive transductions is the key tool in 
our further analysis. It will be very useful to give another (seemingly) 
weaker property for the existence of an immersive transduction in another
class, 
%
which is the existence of a transduction that does not shrink the distances too much, as we prove now.

\begin{lemma}
\label{lem:epsilon}
	Assume there is a non-copying transduction $\mathsf T$ encoding $\Cc$ in $\Dd$ with associated interpretation $\mathsf I$ and an $\epsilon>0$ with the property that 
	for every $H\in\Cc$ and $G\in\Dd$ with $H\in \mathsf T(G)$ we have
	 ${\rm dist}_{H}(u,v)\geq \epsilon\,{\rm dist}_G(u,v)$ (for all $u,v\in V(H)$).
%
	 Then there exists an immersive transduction encoding $\Cc$ in $\Dd$ that subsumes $\mathsf T$.
\end{lemma}
\begin{proof}
	Let $\mathsf T=\mathsf I\circ\Gamma_\Sigma$ with $\mathsf I=(\nu(x),\eta(x,y))$. 
	By Gaifman's locality theorem, there is a set $\Sigma'\supseteq\Sigma$ of unary relations and 
	a formula $\phi(x,y)$, such that for every $\Sigma$-expanded graph $G^+$ there is a $\Sigma'$-expansion $G^*$ of $G^+$ with $G^*\models\phi(x,y)$ if and only if $G^+\models\eta(x,y)$, where $\phi$ is $t$-local for some $t$ (as in the proof of \cref{thm:normal}).
	We further define a new mark $M$ and let 
	$\mathsf I'=(M(x),\phi(x,y)\wedge \dist(x,y)\leq 1/\epsilon)$. 	
	The transduction $\mathsf T'=\mathsf I'\circ\Gamma_{\Sigma'\cup\{M\}}$ is immersive and subsumes the transduction~$\mathsf T$.
\end{proof}

Recall that $G\join K_1$ is obtained from $G$ by adding a 
new vertex, called an \emph{apex}, that is connected to all vertices of 
$G$. Of course, by adding an apex we shrink all distances in $G$. The
next lemma shows that when we can transduce $\Cc\join K_1$ in a 
class $\Ff$ with an immersive transduction, then we can in fact transduce 
$\Cc$ in the local balls of 
$\Ff$.

\begin{lemma}
	\label{cor:apex}
Let $\Cc,\Ff$ be graph classes, and let $\mathsf T$ be an 
immersive transduction encoding a class $\Dd$ in $\Ff$ with
$\Dd\supseteq\{G\join K_1\mid G\in\Cc\}$.
Then there exists an integer $r$ such that $\Cc\sqsubseteq_\FO^\circ\rloc{r}{\Ff}$.
\end{lemma}
\begin{proof}
 Let $\mathsf T=\mathsf I\circ\Gamma_\Sigma$ be an immersive transduction encoding $\Dd$ in $\Ff$.
		For every graph $G\in\Cc$ there exists a graph $F\in\Ff$ such that $G\join K_1=\mathsf I(F^+)$, where~$F^+$ is a $\Sigma$-expansion of $F$. Let $v$ be the apex of $G\join K_1$.
		By the strong locality of $\mathsf I$ we get $\mathsf I(F^+)=\mathsf I(B_r^{F^+}(v))$ for some fixed $r$ depending only on $\mathsf T$. Let $\mathsf U$ be a transduction allowing to take an induced subgraph, then $G$ can be encoded 
		in the class $\rloc{r}{\Cc}$ by the non-copying transduction $\mathsf U\circ\mathsf T$. 
\end{proof}


Finally, we show that when transducing an additive class $\Cc$ in a class $\Dd$, then we do not need perturbations at all. 

\begin{lemma}
\label{lem:slunion}
	Let $\Cc$ be an additive  class with $\Cc\sqsubseteq_\FO^\circ\Dd$.  Then there exists an immersive transduction encoding $\Cc$ in $\Dd$.
\end{lemma}
\begin{proof}
	According to \cref{thm:normal}, the transduction of $\Cc$ in $\Dd$ is subsumed by the composition of an immersive transduction $\mathsf T$ (with associated interpretation $\mathsf I=(\nu,\eta)$) and a perturbation (with associated interpretation $\mathsf I_P$). As $\eta$ is strongly local there exists 
	$r$ such that for all $G\in \Dd$ and $\Sigma_\mathsf{T}$-expansions 
	$G^+$ and all $u,v\in \mathsf{I}(G^+)$ we have $\dist_{\mathsf{I}(G^+)}(u,v)\geq
	\dist_G(u,v)/r$. 
	Let $c$ be the number of unary relations used in the perturbation. 
	 Let $H$ be a graph in~$\Cc$, let $n>3\cdot c^{|H|}$ and let $K=nH$ ($n$ disjoint copies of $H$). By assumption there exists an expansion~$G^+$ of a graph~$G$ in $\Dd$ with $K=\mathsf I_P\circ\mathsf I(G^+)$. By the choice of $n$, at least $3$ copies $H_1,H_2$, and~$H_3$ of $H$ in~$K$ satisfy the same unary predicates at the same vertices. 
	 For $a\in\{1,2,3\}$ and $v\in V(H_1)$, we denote by~$\tau_a(v)$ the vertex of $H_a$ corresponding to the vertex~$v$ of $H_1$ ($\tau_1(v)$ being the 
	 vertex $v$ itself).	 
	 Let $u,v$ be adjacent vertices of $H_1$. Assume that $u$ and $v$
	 have distance greater than~$r$ in $G$. Then $u$ and $v$ are made
	 adjacent in $K$ by the perturbation $\mathsf{P}$ (the edge cannot have
	 been created by $\eta$ as it is strongly $r$-local). As $\tau_a(u)$ is not
	 adjacent with $\tau_b(v)$ for $b\neq a$ there must be paths of length
	 at most $r$ linking~$\tau_a(u)$ with~$\tau_b(v)$ in $G$ for $a\neq b$ (the 
	 interpretation $\mathsf{I}$ must have introduced an edge that the
	 perturbation removed again). This however implies that there is a 
	 path of length at most $3r$ between $u$ and $v$ in $G$ (going 
	 from $u$ to $\tau_2(v)$ to $\tau_3(u)$ to $v$). It follows that for all 
	 $u,v\in V(K)$ we have $\dist_K(u,v)\geq \dist_G(u,v)/(3r)$. 
%
Hence the transduction obtained by composing $\mathsf T$ with the extraction of the induced subgraph $H_1$ implies the existence of an immersive transduction of $\Cc$ in $\Dd$, according to \cref{lem:epsilon}. 
\end{proof}
\begin{corollary}[Elimination of the perturbation]\label{crl:slunion}
	Let $\Cc$ be an additive class with  $\Cc\sqsubseteq_\FO\Dd$.  Then there exists a copy operation $\mathsf C$ and an immersive transduction~$\mathsf T_{\rm imm}$ such that $\mathsf T_{\rm imm}\circ\mathsf C$ is a transduction encoding $\Cc$ in $\Dd$.
\end{corollary}

We finish our study of the local properties of transductions by 
considering classes that locally have a certain property, such as classes
with locally bounded treewidth. 
Formally, let~$f$ be a graph invariant, that is, an isomorphism invariant mapping from graphs to natural numbers. A hereditary class $\Cc$ has {\em locally bounded $f$} if there exists a function $g$  with $f(G)\leq g(r)$ for every $G\in\Cc$ with radius at most $r$.
The invariant $f$ is {\em apex-friendly} if $f(G\join K_1)$ is bounded by a function of $f(G)$.

\begin{lemma}
	Assume $f_1$ and $f_2$ are graph invariants, and that $f_1$ is apex-friendly. Assume that every class  with locally bounded $f_1$ can be encoded by a non-copying transduction in a class  with locally bounded $f_2$. Then every class  with bounded $f_1$ can be encoded by a non-copying transduction in a class with bounded $f_2$.
\end{lemma}
\begin{proof}
	Let $\Cc$ be a class with bounded $f_1$. Let $\Cc'$ be the closure of $\{G\join K_1\mid G\in\Cc\}$ by disjoint union. Then $\Cc'$ has locally bounded $f_1$ (as $f_1$ is apex-friendly). By assumption there exists a class $\Dd$ with locally bounded $f_2$ such that $\Cc'\sqsubseteq_\FO^\circ\Dd$. 
	It follows from \cref{lem:slunion} that there is an immersive  transduction of $\Cc'$ in $\Dd'$ (with interpretation $\mathsf I=(\nu,\eta)$). Thus for every $G\in \Cc$ there exists a vertex coloring $H^+$ of $H\in \Dd$ such that $G\join K_1=\mathsf I(H^+)$. By the strong locality of $\eta$, the radius of $H$ is at most some $r$, thus $f_2(H)\leq C(r)$. We deduce (by composing with the generic induced subgraph transduction) that $\Cc$ is a non-copying transduction of the class of graphs $H$ with $f_2(H)\leq C(r)$.
\end{proof}

The assumption that $f_1$ is apex-friendly is necessary. Consider for instance $f_1$ to be the maximum degree, while $f_2$ is the treewidth. A class with locally bounded maximum degree is simply a class with bounded maximal degree, thus has locally bounded treewidth. However, the class of grids (which has bounded maximum degree) cannot be transduced in any class of bounded tree-width as it cannot be transduced in the class of tree-orders.

\section{Some applications of the local normal form}	
\label{sec:applications-nf}
\subsection{Transductions in paths}
\label{sec:paths}

%% file: path.tex
\begin{theorem}
\label{thm:transpath}
	A class $\Cc$ is {\FO} transduction equivalent to the class of paths if and only if it is a perturbation of a class with bounded bandwidth that contains graphs with arbitrarily large connected components.
\end{theorem}

\begin{proof}
Assume 
$\mathsf{T}$ is a transduction of~$\Cc$ in $\Pp$. 
	According to \cref{thm:normal}, $\mathsf{T}\leq \mathsf P\circ\mathsf T_{\rm imm}\circ\mathsf C_k$, where $k\geq 1$, $\mathsf T_{\rm imm}$ is immersive, and $\mathsf P$ is a perturbation. 
	Observe first that $\mathsf{C}_k(\Pp)$ is included in the class of all 
	subgraphs of the~$(k+1)$-power of paths. 
	By the strong locality property of immersive transductions, 
	every class obtained from $\Pp$ by the composition of a copy operation and an immersive transduction has its image included in the class of all the  subgraphs of the $\ell$-power of paths, for some integer $\ell$ depending only on the transduction, hence, in a class of bounded bandwidth. 
	Conversely, assume that $\Cc$ is a perturbation of a class $\Dd$ containing graphs with bandwidth at most $\ell$ that contains graphs with arbitrarily large connected components. 
	Then $\Dd$ is a subclass of the monotone closure (containing all subgraphs of the class) of the class $\mathcal P^\ell$ of $\ell$-powers of paths, which has bounded star chromatic number. We show in \cref{lem:monotone} in \cref{sec:monotone} that we can obtain the monotone closure of a class with bounded star chromatic number as a transduction. By this result and the 
	observation that taking the $\ell$-power is obviously a transduction, 
	we get that $\Cc\sqsubseteq_\FO\Pp$. To see that vice versa
	$\Pp\sqsubseteq_\FO\Cc$ observe that we can first undo the perturbation
	by carrying out the edge complementations in reverse order. Then we have
	arbitrarily large connected components, which in a graph of bounded
	bandwidth have unbounded diameter. From this we can transduce 
	arbitrarily long paths by extracting an induced subgraph. 
\end{proof}


%% file: local.tex
\subsection{Local monadically stable, straight, and dependent classes}
\label{sec:local}
	A class $\Cc$ is {\em locally monadically dependent} if, for every integer $r$, the class~$\rloc{r}{\Cc}$ is monadically dependent;
		a class~$\Cc$ is {\em locally monadically stable} if, for every integer~$r$, the class~$\rloc{r}{\Cc}$ is monadically stable;
		a class~$\Cc$ is {\em locally monadically straight} if, for every integer~$r$, the class~$\rloc{r}{\Cc}$ is monadically straight. 
\begin{theorem}
\label{thm:local}
For a class $\Cc$ of graphs we have the following equivalences:
\begin{enumerate}
\item $\Cc$ is locally monadically dependent if and only if $\Cc$ is
  monadically dependent;
\item $\Cc$ is locally monadically straight if and only if $\Cc$ is
  monadically straight;
\item $\Cc$ is locally monadically stable if and only if $\Cc$ is
  monadically stable.
\end{enumerate}
\end{theorem}

\begin{proof}
	The proof will follow from the following claim.
		\begin{claim}
\label{cl:golocal}
	Let $\Cc$ be a class such that the class $\Cc'=\{n(G\join K_1)\mid n\in\mathbb N, G\in\Cc\}$ is a transduction of $\Cc$. Then, for every class $\Dd$ we have
	$\Cc\sqsubseteq_\FO \Dd$ if and only if there exists some integer $r$ with $\Cc\sqsubseteq_\FO \rloc{r}{\Dd}$.
\end{claim}
\begin{claimproof}
	Obviously, if there exists some integer $r$ with $\Cc\sqsubseteq_\FO \rloc{r}{\Dd}$, then $\Cc\sqsubseteq_\FO \Dd$.
	Now assume  $\Cc\sqsubseteq_\FO\Dd$. As $\Cc'\equiv_\FO\Cc$ we prove as in \cref{lem:slunion} that there is a transduction of~$\Cc'$ in $\Dd$ that is the composition of a copy operation $\mathsf C$ and an immersive transduction $\mathsf T$. Let $\Dd'=\mathsf C(\Dd)$. According to \cref{cor:apex}, there is an integer $r$ such that $\Cc\sqsubseteq_\FO^\circ \rloc{r}{\Dd'}$ thus, as $\rloc{r}{\Dd'}=\mathsf C(\rloc{r}{\Dd})$, we have
	$\Cc\sqsubseteq_\FO \rloc{r}{\Dd}$.
\end{claimproof}

		The class $\{n(G\join K_1)\mid n\in\mathbb N, G\in\Gg\}$ is obviously a transduction of~$\Gg$. Hence, according to \cref{cl:golocal}, 
		a class $\Cc$ is locally monadically dependent if and only if it is monadically dependent.
	The class $\{n(G\join K_1)\mid n\in\mathbb N, G\in\TP\}$ is a  transduction of~$\TP$. Hence, according to \cref{cl:golocal}, 
		a class $\Cc$ is locally monadically straight if and only if it is monadically straight.
		The class $\{n(G\join K_1)\mid n\in\mathbb N, G\in\Hh\}$ is a  transduction of~$\Hh$.
	Hence, according to \cref{cl:golocal}, 
		a class $\Cc$ is locally monadically stable if and only if it is monadically stable.	
\end{proof}

\begin{example}
Although the class of unit interval graphs has unbounded clique-width,
	every proper hereditary subclass of unit interval graphs
	has bounded clique-width~\cite{UIcw}. This is in particular the case for the class of unit interval graphs with bounded radius. As classes with bounded clique-width are monadically dependent, the class of unit interval graphs is locally monadically dependent, hence monadically dependent. 
%
\end{example}

%% file: order.tex
\section{Structural properties of the transduction quasiorders}
\label{sec:structural-prop}
Many properties will be similar when considering $\sqsubseteq_\FO$ and $\sqsubseteq_\FO^\circ$. To avoid unnecessary repetitions of the statements and arguments, we shall use
the notations $\sqsubseteq,\sqsubset,\equiv$ to denote either $\sqsubseteq_\FO^\circ,\sqsubset_\FO^\circ,\equiv_\FO^\circ$ or $\sqsubseteq_\FO,\sqsubset_\FO,\equiv_\FO$.

For two classes $\Cc_1$ and $\Cc_2$ define $\Cc_1+\Cc_2=\{G_1\cup G_2: G_1\in\Cc_1, G_2\in\Cc_2\}$. A class $\Cc$ is \emph{additive} if $\Cc+\Cc\equiv\Cc$, for instance every class closed under disjoint union is additive.  Note that if $\Cc_1$ and $\Cc_2$ are additive then $\Cc_1+\Cc_2$ is also additive. We further say that a class $\Cc$ is \emph{essentially connected} if it is equivalent to the subclass ${\rm Conn}(\Cc)$ of all its connected graphs.

In this section we will consider the quasiorders $\sqsubseteq_\FO^\circ$ and $\sqsubseteq_\FO$, as well as their restrictions to additive classes of graphs.
Let $\mathfrak C$ be the collection of all graph classes, and let $\mathfrak A$ be the collection of all additive graph classes.
%
While speaking about these quasiorders,  we will implicitly consider their quotient by the equivalence relation $\equiv$, which are partial orders. For instance, when we say that  $(\mathfrak C,\sqsubseteq)$ is a join-semilattice, we mean that  $(\mathfrak C\,/\equiv,\sqsubseteq)$ is a join semilattice.
The symbol $\cover$ will always been used with reference to $(\mathfrak C,\sqsubseteq)$, $\Cc\cover\Dd$ expressing that there exist no class $\Ff$ with $\Cc\sqsubset\Ff\sqsubset\Dd$. When we shall consider covers in $(\mathfrak A,\sqsubseteq)$ we will say explicitly that $(\Cc,\Dd)$ is a cover in $(\mathfrak A,\sqsubseteq)$, expressing that there exists no additive class $\Ff$ with $\Cc\sqsubset\Ff\sqsubset\Dd$.

\subsection{The transduction semilattices $(\mathfrak C,\sqsubseteq_\FO^\circ)$ and $(\mathfrak C,\sqsubseteq_\FO)$}
The aim of this section is to prove that $(\mathfrak C,\sqsubseteq_\FO^\circ)$ and $(\mathfrak C,\sqsubseteq_\FO)$ are distributive join-semilattices and to state some of their properties.

\begin{lemma}
	\label{cl:union}
 If $\Dd\sqsubseteq\Cc_1\cup \Cc_2$, then there is a partition $\Dd_1\cup \Dd_2$ of $\Dd$ with $\Dd_1\sqsubseteq\Cc_1$ and 	$\Dd_2\sqsubseteq\Cc_2$.
 	If  $\Dd$ is additive, then
	$\Dd\sqsubseteq\Cc_1\cup \Cc_2\ \iff\ \Dd\sqsubseteq\Cc_1\text{ or }\Dd\sqsubseteq\Cc_2$.
\end{lemma}
\begin{proof}
The first statement is straightforward.
We now prove the second statement.
For an integer $n$, let $G_n$ be the disjoint union of all the graphs in $\Dd$ with at most $n$ vertices.

Assume $\Dd$ is additive and $\Dd\sqsubseteq\Cc_1\cup\Cc_2$.
According to the first statement, there exists a partition $\Dd_1,\Dd_2$ of $\Dd$ with 
$\Dd_1\sqsubseteq\Cc_1$ and $\Dd_2\sqsubseteq\Cc_2$.
For $G\in\Dd$ define \mbox{$\mathscr S(G)=\{H\cup G: H\in\Dd\}$}. Note that $\mathscr S(G)\subseteq\Dd+\Dd$.
Let $\Dd'=\Dd+\Dd$. As $\Dd'\sqsubseteq\Dd_1\cup\Dd_2$ there exists a partition $\Dd_1',\Dd_2'$ of $\Dd'$ with $\Dd_1'\sqsubseteq\Dd_1$ and $\Dd_2'\sqsubseteq\Dd_2$.
If, for every $G\in\Dd$ we have $\mathscr S(G)\cap \Dd_1'\neq\emptyset$ then $\Dd\sqsubseteq \Dd_1'$ (by the generic transduction extracting an induced subgraph) thus $\Dd\equiv\Dd_1\sqsubseteq\Cc_1$. Similarly, if for every $G\in\Cc$ we have $\mathscr S(G)\cap \Dd_2'\neq\emptyset$ then $\Dd\sqsubseteq\Cc_2$. Assume for contradiction that there exist $G_1,G_2\in\Dd$ with $\mathscr S(G_i)\cap \Dd_i'=\emptyset$. Then $G_1\cup G_2$ belongs neither to $\Dd_1'$ nor to $\Dd_2'$, contradicting the assumption that $\Dd_1',\Dd_2'$ is a partition of $\Dd'=\Dd+\Dd$.
\end{proof}

\begin{lemma}
\label{lem:irred1}
 If $\Cc_1$ and $\Cc_2$ are incomparable, then $\Cc_1\cup\Cc_2$ is not equivalent to an additive class. 
  In particular, $\Cc_1\cup\Cc_2\not\equiv\Cc_1+\Cc_2$. 
 \end{lemma}
 \begin{proof}
 	We prove by contradiction that $\Cc_1\cup\Cc_2$ is not equivalent to an additive class.
	Assume that we have $\Dd\sqsubseteq \Cc_1\cup\Cc_2$, where $\Dd$ is additive. According to \cref{cl:union} we have $\Dd\sqsubseteq \Cc_1$ or $\Dd\sqsubseteq \Cc_2$ thus if $\Cc_1\cup\Cc_2\sqsubseteq\Dd$, then 
	$\Cc_2\sqsubseteq \Cc_1$ or $\Cc_1\sqsubseteq \Cc_2$, contradicting the hypothesis that $\Cc_1$ and $\Cc_2$ are incomparable.
 \end{proof}

\begin{theorem}
\label{thm:sl}
	The quasiorder $(\mathfrak C,\sqsubseteq)$ is a distributive join-semilattice, where the join of $\Cc_1$ and $\Cc_2$ is $\Cc_1\cup\Cc_2$. In this quasiorder, additive classes are join-irreducible.
	This quasiorder has a minimum $\Ee$ and a maximum $\Gg$.
\end{theorem}
\begin{proof}
	Of course we have $\Cc_1\sqsubseteq \Cc_1\cup\Cc_2$ and $\Cc_2\sqsubseteq \Cc_1\cup\Cc_2$. 
	Now assume $\Dd$ is such that $\Cc_1\sqsubseteq\Dd$ and  $\Cc_2\sqsubseteq\Dd$. 
	Let $\mathsf T_1$ and $\mathsf T_2$ be transductions encoding $\Cc_1$ and $\Cc_2$ in $\Dd$, with associated interpretations $\mathsf I_1=(\nu_1,\eta_1)$ and $\mathsf I_2=(\nu_1,\eta_1)$. 
	 By relabeling the colors, we can assume that the set $\Sigma_1$ of unary relations used by $\mathsf I_1$ is disjoint from the set $\Sigma_2$ of unary relations used by~$\mathsf I_2$. 
	 Without loss of generality, we have $\mathsf T_1=\mathsf I_1\circ\Gamma_{\Sigma_1}\circ\mathsf C$ and 
	 $\mathsf T_2=\mathsf I_2\circ\Gamma_{\Sigma_2}\circ\mathsf C$, where $\mathsf C$ is a copying operation if $\sqsubseteq$ is $\sqsubseteq_\FO$, or the identity mapping if $\sqsubseteq$ is $\sqsubseteq_\FO^\circ$.
	 Let $M$ be a new unary relation. We define the interpretation $\mathsf I=(\nu,\eta)$ by
	$\nu:=\big((\exists v\ M(v))\wedge \nu_1\big)\vee \big(\neg(\exists v\ M(v))\wedge \nu_2\big)$ and $\eta:=\big((\exists v\ M(v))\wedge \eta_1\big)\vee \big(\neg(\exists v\ M(v))\wedge \eta_2\big)$. Let $G\in\Cc_1\cup \Cc_2$. If $G\in\Cc_1$, then there exists a coloring $H^+$ of $H\in\mathsf C(\Dd)$ with $G=\mathsf I_1(H^+)$. We define~$H^*$ as the expansion of~$H^+$ where all vertices also belong to the unary relation~$M$. Then $G=\mathsf I(H^*)$. Otherwise, if $G\in\Cc_2$, then there exists a coloring $H^+$ of $H\in\mathsf C(\Dd)$ with $G=\mathsf I_2(H^+)$ thus
	$G=\mathsf I(H^+)$. As we did not introduce new copying transductions we deduce $\Cc_1\cup\Cc_2\sqsubseteq\Dd$.	
	It follows that~$(\mathfrak C,\sqsubseteq)$ is a join semi-lattice, which is distributive according to \cref{cl:union}.
	
	That additive classes are join-irreducible follows from \cref{lem:irred1}.
\end{proof}

We now state an easy lemma on covers in distributive join-semilattices.
\begin{lemma}
	\label{lem:slcover}
	Let $(X,\leq)$ be a distributive join-semilattice (with join $\vee$). If $a\cover b$ and $b\not\leq a\vee c$, then $a\vee c\cover b\vee c$.
\end{lemma}
\begin{proof}
	Assume $a\vee c\leq x\leq b\vee c$. As $(X,\leq)$ is distributive there exist $b'\leq b$ and $c'\leq c$ with $x=b'\vee c'$. Thus $a\leq a\vee b'\leq b$. As $a\cover b$, either $a=a\vee b'$ (thus $b'\leq a$) and thus $x=a\vee c$, or $a\vee b'=b$ and then $b\vee c\leq a\vee b'\vee c\leq a\vee x\vee c=x\leq b\vee c$ thus $x=b\vee c$. Hence either $a\vee c=b\vee c$ (which would contradict $b\not\leq a\vee c$), or $a\vee c\cover b\vee c$.
\end{proof}

\begin{corollary}
\label{lem:cover}
 If $\Cc_1\cover\Cc_2$ and $\Cc_2\not\sqsubseteq\Cc_1\cup\Dd$, then $\Cc_1\cup\Dd\cover\Cc_2\cup\Dd$.
\end{corollary}
\begin{corollary}
\label{cor:corcover}
	If $\Cc_1\cover\Cc_2$, $\Cc_1\sqsubseteq\Dd$, and $\Cc_2$ and $\Dd$ are incomparable, then $\Dd\cover \Dd\cup\Cc_2$. 
\end{corollary}
\begin{proof}
As $\Cc_2\not\sqsubseteq\Dd$ and $\Cc_1\sqsubseteq\Dd$ we have $\Cc_2\not\sqsubseteq \Dd\cup\Cc_1$.
\end{proof}
\begin{corollary}
\label{cor:covR}
	If $\Cc_1\cover\Cc_2$, $\Cc_2$ and $\Dd$ are incomparable and $\Cc_2$ is additive, then $\Cc_1\cup\Dd\cover\Cc_2\cup\Dd$.
\end{corollary}
\subsection{The transduction semilattices $(\mathfrak A,\sqsubseteq_\FO^\circ)$ and $(\mathfrak A,\sqsubseteq_\FO)$}
The aim of this section is to prove that $(\mathfrak A,\sqsubseteq_\FO^\circ)$ and $(\mathfrak A,\sqsubseteq_\FO)$ are distributive join-semilattices and to state some of their properties.

\begin{lemma}
\label{lem:irred2}
 If $\Cc_1$ and $\Cc_2$ are incomparable, then $\Cc_1+\Cc_2$ is not essentially connected.
 \end{lemma}
 \begin{proof}
	We prove by contradiction that $\Cc_1+\Cc_2$ is not essentially connected.
	It is immediate that ${\rm Conn}(\Cc_1+\Cc_2)\subseteq\Cc_1\cup\Cc_2$. So if $\Cc_1+\Cc_2$ is essentially connected, then  $\Cc_1\cup \Cc_2$ and $\Cc_1+\Cc_2$ are equivalent, contradicting \cref{lem:irred1}.
 \end{proof}

\begin{lemma}
\label{lem:plus_sl}
A class $\Dd$ is additive if and only if 
	 for all classes $\Cc_1, \Cc_2$ we have
	\[
	\Cc_1+\Cc_2\sqsubseteq\Dd\qquad\iff\quad\Cc_1\sqsubseteq\Dd\text{ and }\Cc_2\sqsubseteq\Dd.
\] 
\end{lemma}
\begin{proof}
Assume $\Dd$ is additive.
	If $\Cc_1+\Cc_2\sqsubseteq\Dd$, then $\Cc_1\cup\Cc_2\sqsubseteq\Dd$ thus 
	$\Cc_1\sqsubseteq\Dd$ and $\Cc_2\sqsubseteq\Dd$. Conversely, assume 
		 $\Cc_1\sqsubseteq\Dd$ and $\Cc_2\sqsubseteq\Dd$.
	Then $\Cc_1\cup \Cc_2\sqsubseteq\Dd$. Let $\mathsf T=\mathsf I\circ\Gamma_\Sigma\circ\mathsf C$ be a transduction such that $\Cc_1\cup\Cc_2\subseteq\mathsf T(\Dd)$, where $\mathsf I=(M(x),\phi(x,y))$ with $M\in\Sigma$, and where~$\mathsf C$ is either a copying operation if $\sqsubseteq$ is $\sqsubseteq_\FO$, or the identity mapping if $\sqsubseteq$ is $\sqsubseteq_\FO^\circ$.
	Let $\Sigma'$ be the signature obtained from $\Sigma$ by adding two unary predicates $A(x)$ and $B(x)$.
	We define $\phi_A(x,y)$ (resp. $\phi_B(x,y)$) by replacing in $\phi(x,y)$ the predicate $M$ by the predicate $A$ (resp. by the predicate $B$). Let $\phi'(x,y)=\bigl(A(x)\wedge A(y)\wedge \phi_A(x,y)\bigr)\vee \bigl(B(x)\wedge B(y)\wedge\phi_B(x,y)\bigr)$, let $\mathsf I'=(A(x)\vee B(x),\phi'(x,y))$, and let $\mathsf T=\mathsf I'\circ\Gamma_{\Sigma'}\circ\mathsf C$. Then it is easily checked that $\Cc_1+\Cc_2\subseteq\mathsf T'(\Dd)$.
Conversely, assume that for all classes $\Cc_1, \Cc_2$ we have $\Cc_1+\Cc_2\sqsubseteq\Dd\ \iff\ \Cc_1\sqsubseteq\Dd\text{ and }\Cc_2\sqsubseteq\Dd$. Then (by choosing $\Cc_1=\Cc_2=\Dd$) we deduce $\Dd+\Dd\equiv\Dd$.
\end{proof}

\begin{lemma}
\label{lem:distributive}
Assume $\Dd$ is additive and  $\Cc_1$ and $\Cc_2$ are incomparable.
If $\Dd\sqsubseteq\Cc_1+\Cc_2$, then there exist  classes $\Dd_1$ and $\Dd_2$ such that 
$\Dd\equiv\Dd_1+\Dd_2$, $\Dd_1\sqsubseteq\Cc_1$ and $\Dd_2\sqsubseteq\Cc_2$. Moreover, if~$\Cc_1$ and $\Cc_2$ are additive we can require that $\Dd_1$ and $\Dd_2$ are additive.
\end{lemma}
\begin{proof}
 According to \cref{crl:slunion} there exists a copy operation $\mathsf C$ (which reduces to the identity if $\sqsubseteq$ is $\sqsubseteq_\FO^\circ$) and an immersive transduction $\mathsf T_{\rm imm}$ such that $\Dd\subseteq\mathsf T_{\rm imm}\circ\mathsf C(\Cc_1+\Cc_2)$. Let $\mathsf I$ be the interpretation part of $\mathsf T_{\rm imm}$.
	Let $G\in\Dd$ and let $H^+$ be a coloring of $H=\mathsf C(K)$, with $K\in\mathsf \Cc_1+\Cc_2$ and $G=\mathsf I(H^+)$.
	As $\mathsf T_{\rm imm}$ is immersive, each connected component of $G$ comes from a connected component of $H^+$ hence from a connected component of $K$. By grouping the connected components used in $K$ by their origin ($\Cc_1$ or $\Cc_2$) we get that $G$ is the disjoint union of $G_1\in \mathsf T_{\rm imm}\circ C(K_1)$ and $G_2\in \mathsf T_{\rm imm}\circ C(K_2)$, where $K_1\in\Cc_1$ and $K_2\in\Cc_2$. 
	So $\Dd\sqsubseteq\Dd_1+\Dd_2$, where $\Dd_1\sqsubseteq\Cc_1$ and $\Dd_2\sqsubseteq\Cc_2$.
	Moreover, as obviously $\Dd_1\sqsubseteq\Dd$ and $\Dd_2\sqsubseteq\Dd$ we derive from \cref{lem:plus_sl} that we have $\Dd_1+\Dd_2\sqsubseteq\Dd$. Hence $\Dd\equiv\Dd_1+\Dd_2$.
	For $i=1,2$, if~$\Cc_i$ is additive, then we can assume that $\Dd_i$ is also additive.
		\end{proof}
\begin{corollary}
\label{cor:ec}
	If $\Dd$ is additive and essentially connected, then
	\[
	\Dd\sqsubseteq\Cc_1+\Cc_2\quad\iff\quad\Dd\sqsubseteq\Cc_1\text{ or }\Dd\sqsubseteq\Cc_2.
	\]
\end{corollary}
\begin{proof}
	According to \cref{lem:distributive} there exist $\Dd_1,\Dd_2$ with $\Dd\equiv\Dd_1+\Dd_2$, $\Dd_1\sqsubseteq\Cc_1$ and $\Dd_2\sqsubseteq\Cc_2$. However, as $\Dd$ is essentially connected, $\Dd_1$ and $\Dd_2$ cannot be incomparable. Thus $\Dd\sqsubseteq\Cc_1$ or $\Dd\sqsubseteq\Cc_2$.
\end{proof}

\begin{theorem}
\label{thm:slA}
	The quasiorder $(\mathfrak A,\sqsubseteq)$ is a distributive join-semilattice, where the join of $\Cc_1$ and $\Cc_2$ is $\Cc_1+\Cc_2$. In this quasiorder, essentially connected (additive) classes are join-irreducible. 	This quasiorder has a minimum $\Ee$ and a maximum $\Gg$.
\end{theorem}
\begin{proof}
	That $(\mathfrak A,\sqsubseteq)$ is a  join-semilattice follows from \cref{lem:plus_sl}; that it is distributive follows from \cref{lem:distributive}. The last statement follows from \cref{lem:irred2}.
\end{proof}

\begin{corollary}
	Assume that $\Cc_1$ and $\Cc_2$ are incomparable and additive, $\Dd$ is additive and essentially connected, $\Cc_1\sqsubseteq\Dd$, and $\Cc_2\sqsubseteq\Dd$. Then we have
\[
\Cc_1\cup\Cc_2\sqsubset\Cc_1+\Cc_2\sqsubset\Dd.
\]
\end{corollary}

\begin{corollary}
\label{cor:incomp}
	Assume that $\Cc_1$ and $\Cc_2$ are incomparable and additive, $\Dd$ is additive and essentially connected, $\Dd$ is incomparable with $\Cc_1$ and $\Cc_2\sqsubset\Dd$.
	Then $\Cc_1+\Cc_2$ is incomparable with $\Dd$.
\end{corollary}
\begin{proof}
	Assume for contradiction that $\Cc_1+\Cc_2\sqsubseteq\Dd$.
	According to \cref{thm:slA}, we have $\Cc_1\sqsubseteq\Dd$, contradicting the assumption that $\Dd$ is incomparable with $\Cc_1$.
	
	Assume for contradiction that $\Dd\sqsubseteq\Cc_1+\Cc_2$. 
	According to \cref{thm:slA} there exists (by distributivity) classes $\Dd_1$ and $\Dd_2$ with $\Dd_1\sqsubseteq\Cc_1$, $\Dd_2\sqsubseteq\Cc_2$, and $\Dd=\Dd_1+\Dd_2$. As $\Dd$ is essentially connected, according to \cref{thm:slA} it is join-irreducible. Thus $\Dd_1$ and $\Dd_2$ are comparable. 
	Thus  $\Dd\subseteq\Cc_1$ or $\Dd\subseteq\Cc_2$. The first case does not hold as $\Dd$ is incomparable with $\Cc_1$, and the second case does not hold as $\Cc_2\sqsubset\Dd$.
\end{proof}


Using the distributive join-semillatice structure of $(\mathfrak A,\sqsubseteq)$, the following corollaries follow from \cref{lem:slcover}.
\begin{corollary}
\label{lem:coverA}
 In the poset $(\mathfrak A,\sqsubseteq)$, if $(\Cc_1,\Cc_2)$ is a cover and $\Cc_2\not\sqsubseteq\Cc_1+\Dd$  then $(\Cc_1+\Dd$, $\Cc_2+\Dd)$ is a cover.
\end{corollary}
\begin{corollary}
\label{cor:corcoverA}
	If $(\Cc_1,\Cc_2)$ is a cover of $(\mathfrak A,\sqsubseteq)$, $\Cc_1\sqsubseteq\Dd$, and $\Cc_2$ and $\Dd$ are incomparable, then $(\Dd,\Dd+\Cc_2)$ is a cover of $(\mathfrak A,\sqsubseteq)$. 
\end{corollary}
\begin{corollary}
\label{cor:covRA}
	If $(\Cc_1,\Cc_2)$ is a cover of $(\mathfrak A,\sqsubseteq)$, $\Cc_2$ and $\Dd$ are incomparable, and $\Cc_2$ is essentially connected, then
	$(\Dd,\Dd+\Cc_2)$ is a cover of $(\mathfrak A,\sqsubseteq)$.
\end{corollary}

\section{The transduction quasiorder on some classes}\label{sec:special}
In this section we consider the  poset $(\mathfrak C,\sqsubseteq_\FO)$.
We focus on the structure of the partial order in the region of classes with bounded tree-width.
A schematic view of the structure of $(\mathfrak C,\sqsubseteq_\FO)$ on classes with tree-width at most $2$ is shown \cref{fig:TW2}

Recall that since {\MSO} collapses to {\FO} on colored trees of
bounded depth we have the following chain of covers
$\Ee\cover_\FO
\Tt_2\cover_\FO\Tt_3\cover_\FO\ldots$. We first prove
that parallel to this chain we have a chain of covers
$\Ee\cover_\FO\Pp\cover_\FO\Pp\cup\Tt_2\cover_\FO\Pp\cup\Tt_3\cover_\FO\dots$
and for all $n\geq 1$ we have $\Tt_n\cover_\FO\Pp\cup\Tt_n$.


\begin{restatable}[\blah]{theorem}{covers}
\label{thm:covers}
We have $\Ee\cover_\FO\Pp$ and, for every $n\geq 1$, the chain of covers
\[(\Pp+\Tt_n)\cover_\FO (\Pp+\Tt_n)\cup\Tt_{n+1}\cover_\FO(\Pp+\Tt_n)\cup\Tt_{n+2}\cover_\FO\dots\]
In particular, for $n=1$ we get
$\Pp\cover_\FO \Pp\cup\Tt_{2}\cover_\FO\Pp\cup\Tt_{3}\cover_\FO\dots$

\noindent Moreover, for all $n\geq 1$ we have $\Tt_n\cover_\FO\Pp\cup\Tt_n$.
\end{restatable}

The difficult part of the next theorem is to prove that $\Tt_{n+2}\not\sqsubseteq_\FO\PW_n$. We use that the
class $\Tt_{n+2}$ is additive, which by \cref{crl:slunion}
implies that we can eliminate perturbations and focus on immersive
transductions.
This allows us to consider host graphs in $\PW_n$ that have
bounded radius, where we can find a small set of vertices whose
removal decreases the pathwidth. We encode the adjacency 
to these vertices by colors and proceed by induction. 

\begin{restatable}[\blah]{theorem}{thmPW}
\label{thm:PW}
  For $n\geq 1$ we have $\Tt_{n+1}\sqsubset_\FO\PW_n$ but
  \mbox{$\Tt_{n+2}\not\sqsubseteq_\FO\PW_n$}.
  Consequently, for $m>n\geq 1$ we have
  
  \begin{align*}
  \Tt_m+\PW_n\cover_\FO (\Tt_m+\PW_n)\cup \Tt_{m+1}
&  \cover_\FO (\Tt_m+\PW_n)\cup \Tt_{m+2}\cover_\FO \dots\\
 \Tt_m+\PW_n\cover_\FO (\Tt_m+\PW_n)\cup \Tt_{m+1}&\sqsubset_\FO\Tt_{m+1}+\PW_n	
  \end{align*}
  In particular, fixing $m=n+1$ we get that for $n\geq 1$ we have
  \begin{align*}
\PW_n\cover_\FO \PW_n\cup \Tt_{n+2}&\cover_\FO \PW_n\cup \Tt_{n+3}\cover_\FO \dots\sqsubset_\FO\PW_{n}\cup\Tt\\
\PW_n\cover_\FO \PW_n\cup \Tt_{n+2}&\sqsubset_\FO\Tt_{n+2}+\PW_n\sqsubset_\FO\PW_{n+1}	
  	  \end{align*}
%
\end{restatable}

\begin{restatable}[\blah]{theorem}{PWTW}
\label{thm:PWTW}
For $m>n\geq 2$ , $\Tt_m+\PW_n$ is incomparable with~$\Tt$.
Consequently, we have 
\[\Tt\sqsubset_\FO\Tt\cup\PW_2\sqsubset_\FO\Tt\cup(\Tt_4+\PW_2)\sqsubset_\FO\dots\sqsubset_\FO\Tt+\PW_2\sqsubset_\FO\TW_2.\]
\end{restatable}

With the above results in hand we obtain for $(\mathfrak C,\sqsubseteq_\FO)$ and $(\mathfrak A,\sqsubseteq_\FO)$ the structures sketched in \cref{fig:TW2}.

\begin{figure}[h!t]
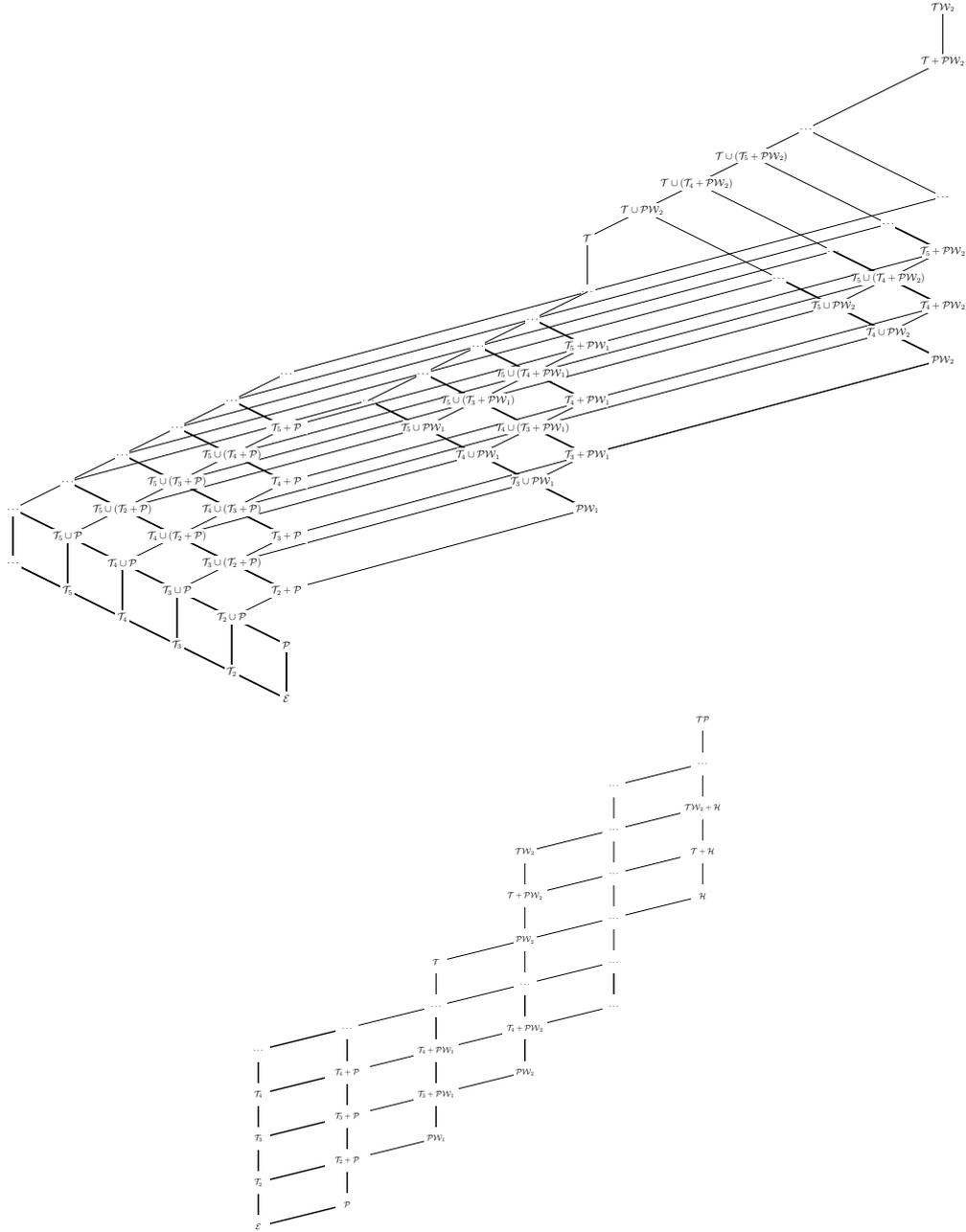

\begin{center}
	\includegraphics[height=.44\textheight]{arbo}	\\
		\includegraphics[height=.33\textheight]{arbo2}
\end{center}
	\caption{A fragment of $(\mathfrak C,\sqsubseteq_\FO)$ (top) and a fragment of $(\mathfrak A,\sqsubseteq_\FO)$ (bottom). Thick edges are covers.}
	\label{fig:TW2}
\end{figure}

%% file: monotone.tex
\section{Missing proofs from \cref{sec:special}}
\begin{lemma}
\label{lem:Ee}
	A class $\Cc$ is a transduction in $\Ee$ (or, equivalently, $\Cc\equiv_\FO\Ee$) if and only if $\Cc$ is a perturbation of a class whose members have uniformly bounded size connected components. Furthermore, a class $\Cc$ is a non-copying transduction in $\Ee$ if and only if $\Cc$ is a perturbation of $\Ee$.
\end{lemma}
\begin{proof}
	Assume that $\Cc$ is a perturbation of a class $\Dd$ and that all graphs in $\Dd$ have connected components of size at most $n$. Then $\Dd$ can obviously be obtained from $\Ee$ by the composition of an $n$-copy transduction composed with some (immersive) transduction. 
	Conversely, if $\Cc$ is a transduction of $\Ee$ then, according to \cref{thm:normal}, $\Cc$ is a perturbation of a class $\Dd\subseteq\mathsf T_{\rm imm}\circ \mathsf C(\Ee)$, where $\mathsf T_{\rm imm}$ is immersive and $\mathsf C$ is a $k$-copy transduction for some number $k$. By the strong locality of the interpretation associated to $\mathsf T_{\rm imm}$ no connected component of a graph in $\Dd$ can have size greater than $k$.
	The last statement is obvious, as every immersive transduction of $\Ee$ is included in $\Ee$.
\end{proof}

\covers*
\begin{proof}
Obviously $\Ee\sqsubset_\FO\Pp$. Assume towards a contradiction that 
there exists a class $\Cc$ with $\Ee\sqsubset_\FO\Cc\sqsubset_\FO\Pp$.
As in the first part of the proof of \cref{thm:transpath}, we see that 
the class~$\Cc$ is a perturbation of a class $\Dd$ with bounded bandwidth.  
The connected components of the graphs in $\Dd$ have unbounded size as otherwise $\Dd$ (and thus $\Cc$) can be transduced in $\Ee$ (according to \cref{lem:Ee}). 
 But then $\Cc\equiv_\FO \Pp$ by \cref{thm:transpath}, 
contradicting our assumption. 
 Hence~$\Ee\cover_\FO\Pp$. 
 
		Let $m\geq n\geq 1$. We have $(\Pp+\Tt_n)\not\sqsubseteq_\FO\Tt_{m+1}$ and 
		$\Tt_{m+1}\not\sqsubseteq_\FO(\Pp+\Tt_n)\cup\Tt_m$ as otherwise, according to \cref{cl:union} we would have $\Tt_{m+1}\sqsubseteq_\FO\Pp+\Tt_n$ or $\Tt_{m+1}\sqsubseteq_\FO\Tt_m$ as $\Tt_{m+1}$ is additive. As $\Tt_{m+1}$ is additive and essentially connected $\Tt_{m+1}\sqsubseteq_\FO\Pp+\Tt_n$ would imply $\Tt_{m+1}\sqsubseteq_\FO\Pp$ or $\Tt_{m+1}\sqsubseteq_\FO\Tt_n$.
		According to \cref{lem:cover} we have 	$(\Pp+\Tt_n)\cup\Tt_m\cover_\FO(\Pp+\Tt_n)\cup\Tt_{m+1}$. (Note that $\Tt_1=\Ee$ and thus $\Pp\cup\Tt_1=\Pp$, and that $(\Pp+\Tt_n)\cup\Tt_n=\Pp+\Tt_n$.)
		
Let $n\geq 1$.
		According to \cref{cor:corcover}, as $\Ee\cover_\FO\Pp$,  $\Ee\sqsubseteq_\FO\Tt_n$, and $\Pp$ is incomparable with $\Tt_n$ we have $\Tt_n\cover_\FO\Pp\cup\Tt_n$. 
\end{proof}
\thmPW*
\begin{proof}
For convenience, we define $\Tt_1=\PW_0=\Ee$, which is consistent with our definitions.
Then $\Tt_{n}\sqsubseteq_\FO\PW_{n-1}$ as for $n=1$ this reduces to $\Ee\sqsubseteq_\FO\Ee$ and, if $n>1$, $\Tt_{n}\sqsubset_\FO\Pp\cup \Tt_{n}\subset \PW_{n-1}$. We now prove $\Tt_{n+1}\not\sqsubseteq_\FO\PW_{n-1}$
		 by induction on $n$. For $n=1$, this follows from \cref{lem:Ee} (or from the known fact $\Tt_2\sqsupset\Ee$). 
		 Now assume that we have proved the statement for~$n\geq 1$ and assume towards a contradiction that $\Tt_{n+2}\sqsubseteq_\FO\PW_{n}$.
		The class~$\PW_{n}$ is the monotone closure of the class $\mathcal{I}_{n+1}$ of interval graphs with clique number at most~$n+1$. As the class~$\mathcal{I}_{n+1}$ has bounded star chromatic number, it follows from \cref{lem:monotone} that $\PW_{n}\sqsubseteq_\FO\mathcal{I}_{n+1}$. Let~$\mathsf T$ be the composition of the respective transductions, so that  $\Tt_{n+2}\subseteq\mathsf T(\mathcal{I}_{n+1})$.
		As $\Tt_{n+3}$ is additive, 
		it follows from \cref{crl:slunion} that there is a copy operation $\mathsf C$ and an immersive transduction $\mathsf{T}_0$ such that $\mathsf T_0\circ \mathsf C$ subsumes $\mathsf T$.
		As $\{G+K_1: G\in\Tt_{n+1}\}\subseteq \Tt_{n+2}$ if follows from \cref{cor:apex} that there exists an integer $r$ such that $\Tt_{n+1}\sqsubseteq^\circ_\FO \rloc{r}{\mathsf C(\mathcal I_{n+1})}\subseteq \mathsf C(\rloc{r}{\mathcal I_{n+1}})$. According to \cref{lem:slunion}, as $\Tt_{n+1}$ is additive, there exists an immersive transduction $\mathsf{T}_1$ with $\Tt_{n+1}\subseteq \mathsf T_1\circ\mathsf C(\rloc{r}{\mathcal I_{n+1}})$.
		
		Let $G\in \rloc{r}{\mathcal I_{n+1}}$ and let $P$ be a shortest path connecting the 
minimal and maximal vertices in an interval representation of $G$. Then $P$ has length at most $2r$, $P$ dominates $G$, and  $G-P\in\mathcal I_n$. 
		By encoding the adjacencies in~$G$ to the vertices of $P$ by a monadic expansion, we get that that there exists a transduction~$\mathsf T_2$ (independent of our choice of $G$) such that $G$ is a $\mathsf T_2$-transduction of $H$, where $H$ is obtained from $G$ by deleting all the edges incident to a  vertex in $P$. 
		In particular, $\rloc{r}{\mathcal I_{n+1}}\subseteq\mathsf T_2(\mathcal I_n)$ thus
		$\Tt_{n+1}\subseteq\mathsf T_1\circ\mathsf C\circ\mathsf T_2(\mathcal I_n)$, which contradicts our induction hypothesis.

Let us now prove the cover chain, that is that for $k\geq m>n\geq 1$ we have
$(\Tt_m+\PW_n)\cup\Tt_k\cover_\FO (\Tt_m+\PW_n)\cup \Tt_{k+1}$. To deduce this from the cover
$\Tt_k\cover_\FO\Tt_{k+1}$ using \cref{lem:cover} we need to check
$(\Tt_m+\PW_n)\not\sqsubseteq_\FO \Tt_{k+1}\not\sqsubseteq_\FO (\Tt_m+\PW_n)\cup\Tt_k$.
According to \cref{lem:plus_sl}, as $(\Tt_m+\PW_n)$ and $\Tt_{k+1}$ are both additive,
$(\Tt_m+\PW_n)\sqsubseteq_\FO \Tt_{k+1}$ is equivalent to $\Tt_m\sqsubseteq_\FO  \Tt_{k+1}$  and
$\PW_n\sqsubseteq_\FO  \Tt_{k+1}$. However, $\PW_n\not\sqsubseteq_\FO  \Tt_{k+1}$ as $\Pp\sqsubseteq_\FO\PW_n$ and $\Pp$ is incomparable with $\Tt_{k+1}$. Thus $(\Tt_m+\PW_n)\not\sqsubseteq_\FO \Tt_{k+1}$.
Now assume for contradiction that we have $\Tt_{k+1}\sqsubseteq_\FO (\Tt_m+\PW_n)\cup\Tt_k$.
As $\Tt_{k+1}$ is additive, it follows from \cref{cl:union} that we have
$\Tt_{k+1}\sqsubseteq_\FO (\Tt_m+\PW_n)$ (as $\Tt_{k+1}\not\sqsubseteq_\FO Tt_k$).
As $\Tt_{k+1}$ is additive and essentially connected, according to \cref{cor:ec} we deduce that  $\Tt_{k+1}\sqsubseteq_\FO \Tt_m$ or $\Tt_{k+1}\sqsubseteq_\FO\PW_n$. The first case does not hold as $k+1>m$. The second case does not hold as $k+1\geq n+2$ and $\Tt_{n+2}\not\sqsubseteq_\FO\PW_n$.
\end{proof}

\PWTW*
\begin{proof}
	We establish that $\PW_2\not\sqsubseteq_\FO\Tt$, from which $\TW_2\not\sqsubseteq_\FO \Tt$ immediately follows.
	Assume for contradiction that $\PW_2\sqsubseteq_\FO\Tt$. 
As $\PW_2$ is additive and does not need copying (\cref{fact:non-copy-classes}), by \cref{crl:slunion} there is an immersive transduction of $\PW_2$ in $\Tt$. In particular, there exists an integer $r$ such that the class $\{K_1\join P_n: n\in\mathbb N\}$ is the transduction of $\Tt_r$, which contradicts the property 
		$\Pp\not\sqsubseteq_\FO\Tt_r$. Thus  $\PW_2\not\sqsubseteq_\FO\Tt$.

We prove that the classes $\PW_n$ ($n\geq 2$) are incomparable with $\Tt$.
The class $\Tt$ has unbounded linear cliquewidth, hence $\Tt\not\sqsubseteq_\FO\Hh$, and thus \mbox{$\Tt\not\sqsubseteq_\FO\PW_n$} for every integer $n$ (recall that a class $\Cc$ has bounded linear cliquewidth if and only if $\Cc\sqsubseteq_\FO \Hh$). As $\PW_2\not\sqsubseteq_\FO\Tt$ it follows that,  for $n\geq 2$,  $\PW_n\not\sqsubseteq_\FO\Tt$, thus $\PW_n$ and~$\Tt$ are incomparable. 
	(This settles the case where $m=n+1$ as $\Tt_{n+1}+\PW_n=\PW_n$.)
For $m>n+1$ the classes $\Tt_m$ and $\PW_n$ are additive and incomparable and $\Tt_m\sqsubset_\FO\Tt$. As 
$\Tt$ is additive, essentially connected and incomparable with $\PW_n$ we deduce by \cref{cor:incomp} that
$\Tt_m+\PW_n$ is incomparable with $\Tt$.
\end{proof}

\section{Transduction of the monotone closure in a class with bounded star coloring number}
\label{sec:monotone}
Recall that a {\em star coloring} of a graph $G$ is  a proper coloring of $V(G)$ such that any two color classes induce a star forest (or, equivalently, such that no path of length four is $2$-colored), and that the the star chromatic number $\chi_{\rm st}(G)$ of $G$ is the minimum number of colors in a star coloring of $G$~\cite{alon2}. The star chromatic number is bounded on classes of graphs excluding a minor~\cite{Taxi_jcolor} and, more generally, on every bounded expansion class~\cite{nevsetvril2008grad}.

\begin{lemma}
\label{lem:monotone}
Let $\Cc$ be a class with bounded star chromatic number. Then there is an immersive transduction from $\Cc$ onto its monotone closure.
\end{lemma}
\begin{proof} 
	We consider the immersive transduction $\mathsf T$ with $\Sigma_{\mathsf T}=\{M_1,\dots,M_C,N_1,\dots,N_C,X\}$ and $\mathsf I_{\mathsf T}=(\nu(x),\eta(x,y))$, where
	 $C=\max_{G\in\Cc}\chi_{\rm st}(G)$, $\nu(x)=X(x)$, and 
\[
		\eta(x,y)=E(x,y)\wedge\Bigl(\bigvee_{i=1}^C M_i(x)\wedge N_i(y)\Bigr)\wedge \Bigl(\bigvee_{i=1}^C M_i(y)\wedge N_i(x)\Bigr).
\]
	
	Let us prove that $\mathsf T$ is a transduction from $\Cc$ onto its monotone closure.
	Let $H$ be a graph in the monotone closure of $\Cc$ and let $G\in\Cc$ be a supergraph of $H$
	(we identify $H$ with a subgraph of $G$). We consider an arbitrary star coloring $\gamma\colon V(G)\rightarrow [C]$ and define the $\Sigma_{\mathsf T}$-expansion $G^+$ of $G$ as follows:
	\begin{align*}
		X(G)&=V(H),\\
		M_i(G)&=\{v\in V(G)\mid \gamma(v)=i\}&\text{(for $1\leq i\leq C$),}\\
		N_i(G)&=\{v\in V(G)\mid (\exists u\in N_H(v))\ \gamma(u)=i\}&\text{(for $1\leq i\leq C$).}
	\end{align*}
	
	According to the definitions of $\nu$ and as $\models\eta(x,y)\rightarrow E(x,y)$, it is clear that $\mathsf{I}_{\mathsf T}(G^+)$ is a subgraph of $G[V(H)]$. According to our definitions of $M_i$ and $N_i$, it is also clear that~$H$ is a subgraph of $\mathsf{I}_{\mathsf T}(G^+)$. Thus, in order to prove that $H=\mathsf{I}_{\mathsf T}(G^+)$, we have only to prove that~$\mathsf{I}_{\mathsf T}(G^+)$ contains no more edges than $H$. Assume towards a contradiction that $u,v\in V(H)$ are adjacent in $\mathsf{I}_{\mathsf T}(G^+)$ and not in $H$. According to the definition of $M_i$ and $N_i$ it follows that $u$ has a neighbor $v'$ in $H$ with $\gamma(v')=\gamma(v)$ and that  $v$ has a neighbor $u'$ in~$H$ with $\gamma(u')=\gamma(u)$. It follows that the path $v',u,v,u'$ of $G$ is $2$-colored by $\gamma$, contradicting the hypothesis that $\gamma$ is a star coloring.
\end{proof}